\documentclass[a4paper,12pt,reqno]{amsart}
\usepackage{amsfonts}
\usepackage{amsmath}
\usepackage{amssymb}
\usepackage[a4paper]{geometry}
\usepackage{mathrsfs}
\usepackage{hyperref}
\renewcommand\eqref[1]{(\ref{#1})} 
\numberwithin{equation}{section}
\theoremstyle{plain}
\newtheorem{thm}{Theorem}[section]
\newtheorem{prop}[thm]{Proposition}
\newtheorem{cor}[thm]{Corollary}
\newtheorem{lem}[thm]{Lemma}

\theoremstyle{definition}
\newtheorem{defn}[thm]{Definition}
\newtheorem{rem}[thm]{Remark}

\begin{document}

   \title[Layer potentials on homogeneous Carnot groups]
   {Layer potentials, Kac's problem, and refined Hardy inequality on homogeneous Carnot groups}

\author[Michael Ruzhansky]{Michael Ruzhansky}
\address{
  Michael Ruzhansky:
  \endgraf
  Department of Mathematics
  \endgraf
  Imperial College London
  \endgraf
  180 Queen's Gate, London SW7 2AZ
  \endgraf
  United Kingdom
  \endgraf
  {\it E-mail address} {\rm m.ruzhansky@imperial.ac.uk}
  }
\author[Durvudkhan Suragan]{Durvudkhan Suragan}
\address{
  Durvudkhan Suragan:
  \endgraf
  Institute of Mathematics and Mathematical Modelling
  \endgraf
  125 Pushkin str.
  \endgraf
  050010 Almaty
  \endgraf
  Kazakhstan
  \endgraf
  and
  \endgraf
  Department of Mathematics
  \endgraf
  Imperial College London
  \endgraf
  180 Queen's Gate, London SW7 2AZ
  \endgraf
  United Kingdom
  \endgraf
  {\it E-mail address} {\rm d.suragan@imperial.ac.uk}
  }

\thanks{The authors were supported in parts by the EPSRC
 grant EP/K039407/1 and by the Leverhulme Grant RPG-2014-02,
 as well as by the MESRK grant 5127/GF4.}


     \begin{abstract}
     We propose the analogues of boundary layer potentials for the sub-Laplacian on homogeneous Carnot groups/stratified Lie groups and prove continuity results for them. In particular, we show continuity of the single layer potential and establish the Plemelj type jump relations for the double layer potential.
     We prove sub-Laplacian adapted versions of the Stokes theorem as well as of Green's first and second formulae on homogeneous Carnot groups.
     Several applications to boundary value problems are given.
  As another consequence, we derive formulae for traces of the Newton potential for the sub-Laplacian to piecewise smooth surfaces. Using this we construct and study a nonlocal boundary value problem for the sub-Laplacian extending to the setting of the homogeneous Carnot groups M.~Kac's ``principle of not feeling the boundary''. We also obtain similar results for higher powers of the sub-Laplacian.
  Finally, as another application, we prove refined versions of Hardy's inequality and of the uncertainty principle.
     \end{abstract}
     \maketitle
     
     \tableofcontents

\section{Introduction}

The central idea of solving boundary value problems for differential equations in a domain requires the knowledge of the corresponding fundamental solutions, and this idea has a long history dating back to the works of mathematicians such as Gauss \cite{G77, G29} and Green \cite{Gr28}. Nowadays the appearing boundary layer operators and elements of the potential theory serve as major tools for the analysis and construction of solutions to boundary value problems. There is vast literature concerning modern theory of boundary layer operators and potential theory as well as their important applications. In addition, in last decades many interesting and promising works combining the group theory with the analysis of partial differential equations have been presented by many authors. For example, nilpotent Lie groups play an important role in deriving sharp subelliptic estimates for differential operators on manifolds, starting from the seminal paper by Rothschild and Stein \cite{RS}. Moreover, in recent decades, there is a rapidly growing interest for sub-Laplacians on Carnot groups (and also for operators on graded Lie groups), because these operators appear not only in theoretical settings (see e.g. Gromov \cite{Gr96} or Danielli, Garofalo and Nhieu \cite{DGN07} for general expositions from different points of view), but also in application settings such as mathematical models of crystal material and human vision (see, for example, \cite{Chr} and \cite{CMS}). Moreover, sub-Laplacians on homogeneous Carnot groups serve as approximations for general H\"ormander's sums of squares of vector fields on manifolds in view of the Rothschild-Stein lifting theorem \cite{RS} (see also \cite{Folland:RS-CPDE-1977, Rothschild:HN-equiv-CPDE-1983}).

In this paper we discuss elements of the potential theory and the theory of boundary layer operators on homogeneous Carnot groups. As we are not relying on the use of the control distance but on the fundamental solutions everything remains exactly the same (without any changes) if we replace the words `homogeneous Carnot group' by `stratified Lie group'. However, as a larger part of the current literature seems to use the former terminology we also adopt it for this paper.

From a different point of view than ours similar problems have been considered by Folland and Stein \cite{FS}, Geller \cite{Gel90}, Jerison \cite{J}, Romero \cite{R91}, Capogna, Garofalo and Nhieu \cite{CGH08}, Bonfiglioli, Lanconelli and Uguzzoni \cite{BLU07} and a number of other people. A general setting of degenerate elliptic operators was considered by Bony \cite{Bony}.

One of the applications of the analysis of our paper is that we can use elements of the developed potential theory to construct new well-posed (solvable in the classical sense) boundary value problems in addition to using it in solving the known problems such as for Dirichlet and Neumann sub-Laplacians.
Thus, we rely on the developed potential theory to derive trace formulae for the Newton potential of the sub-Laplacian to piecewise smooth surfaces and use these conditions to construct the analogue of Kac's boundary value problem in the setting of homogeneous Carnot groups as well as Kac's ``principle of not feeling the boundary'' for the sub-Laplacian.
As in the classical case, the Kac boundary value problem also serves as an example of a boundary value problem which is explicitly solvable in any domain.

For example, for a bounded domain of the Euclidean space $\Omega\subset\mathbb R^{d},\,\, d\geq 2,$ it is very well known that the
solution to the Laplacian equation in $\mathbb R^{d},$
\begin{equation}
\Delta u(x)=f(x), \,\,\,\ x\in\Omega,\label{15}
\end{equation}
is given by the Green formula (or the Newton potential formula)
\begin{equation}
u(x)=\int_{\Omega}\varepsilon_{d}(x-y)f(y)dy,\,\, x\in\Omega,
\label{14}
\end{equation}
for suitable functions $f$ supported in $\Omega$.
Here $\varepsilon_{d}$ is the fundamental solution to $\Delta$ in $\mathbb R^d$ given by
\begin{equation}
\varepsilon_{d}(x-y)=\left\{
\begin{array}{ll}
    \frac{1}{(2-d)s_{d}}\frac{1}{|x-y|^{d-2}},\,\,d\geq 3,\\
    \frac{1}{2\pi}\log|x-y|, \,\,d=2, \\
\end{array}
\right.
\label{EQ:fs}
\end{equation}
where $s_{d}=\frac{2\pi^{\frac{d}{2}}}{\Gamma(\frac{d}{2})}$ is the surface area of the unit sphere in
$\mathbb R^{d}$.
An interesting  question having several important applications is what boundary condition can be put on $u$ on
the (piecewise smooth) boundary $\partial\Omega$ so that equation \eqref{15} complemented by this
boundary condition would have the solution in $\Omega$ still given by the same formula \eqref{14},
with the same kernel $\varepsilon_d$ given by \eqref{EQ:fs}.
This amounts to finding the trace of the Newton potential \eqref{14}
to the boundary surface $\partial\Omega$.

It turns out that the answer to these questions is the integral boundary condition
\begin{equation}
-\frac{1}{2}u(x)+\int_{\partial\Omega}\frac{\partial\varepsilon_{d}(x-y)}{\partial n_{y}}u(y)d S_{y}-
\int_{\partial\Omega}\varepsilon_{d}(x-y)\frac{\partial u(y)}{\partial n_{y}}d S_{y}=0,\,\,
x\in\partial\Omega,
\label{16}
\end{equation}
where $\frac{\partial}{\partial n_{y}}$ denotes the outer normal
derivative at a point $y$ on $\partial\Omega$.
Thus, the trace of the Newton potential
\eqref{14} on the boundary surface $\partial\Omega$ is determined by \eqref{16}.

The boundary condition \eqref{16} appeared in M. Kac's work \cite{Kac1} where
he called it and the subsequent spectral analysis ``the principle of not feeling the boundary''.
This was further expanded in Kac's book \cite{Kac2}
with several further applications to the spectral theory and the asymptotics of the Weyl's eigenvalue
counting function.
Spectral problems related to the boundary value problem \eqref{15}, \eqref{16} were considered in the papers \cite{KS1}, \cite{KS3}, \cite{Ruzhansky-Suragan-log} and \cite{Ruzhansky-Suragan:NP}.
In general, the boundary value problem \eqref{15}, \eqref{16} has various interesting properties and applications
(see, for example, Kac \cite{Kac1,Kac2} and Saito \cite{Sa}).
The boundary value problem \eqref{15}, \eqref{16} can also be generalised for higher powers of the Laplacian, see \cite{KS3, KS2}.

The analogues of the problem \eqref{15}, \eqref{16} for the Kohn Laplacian
and its powers on the Heisenberg group have been recently investigated by the authors in \cite{Ruzhansky-Suragan:Kohn-Laplacian}, see also \cite{FRcr} for the
more general pseudo-differential analysis in the setting of the Heisenberg group.

One of the aims of this paper is to construct the analogues of boundary layer potentials for the sub-Laplacian on homogeneous Carnot groups and study continuity results for them, as well as to obtain an analogue of the boundary value problem \eqref{15}, \eqref{16} on the homogeneous Carnot groups.

For the convenience of the reader let us now briefly recapture the main results of this paper.
Let $\mathbb{G}$ be a homogeneous Carnot group of homogeneous dimension
$Q\geq 3$ with Haar measure $d\nu$, and let $X_{1},\ldots,X_{N_{1}}$ be left-invariant vector fields giving the first
stratum, with the sub-Laplacian
\begin{equation*}\label{sublap0}
\mathcal{L}=\sum_{k=1}^{N_{1}}X_{k}^{2}.
\end{equation*}
For precise definitions we refer to Section \ref{SEC:prelim}.
Throughout this paper $\Omega\subset\mathbb{G}$ will be an admissible domain:
\begin{defn}\label{DEF:adomain}
We say that an open set $\Omega\subset\mathbb{G}$ is an {\em admissible domain} if it is bounded and if its boundary $\partial\Omega$ is piecewise smooth and simple, that is, it has no self-intersections.
\end{defn}
The condition for the boundary to be simple amounts to $\partial\Omega$ being orientable.
Thus, in this paper:
\begin{itemize}
\item
We establish in Proposition \ref{stokes} the divergence formula (a version of the Stokes theorem)
in the form
$$
\int_{\Omega}\sum_{k=1}^{N_{1}}X_{k}f_{k}d\nu=
\int_{\partial\Omega}\sum_{k=1}^{N_{1}}\langle f_{k}X_{k},d\nu\rangle,
$$
where the form $\langle X_{k},d\nu\rangle$ is obtained from the natural pairing of $d\nu$ (viewed as a form) with $X_{k}$, see \eqref{Xjdnu0}--\eqref{Xk0} for the precise formula.
We obtain the analogue of Green's first formula:
if $v\in C^{1}(\Omega)\bigcap C(\overline{\Omega})$ and $u\in C^{2}(\Omega)\bigcap C^{1}(\overline{\Omega})$, then
\begin{equation} \label{g10}
\int_{\Omega}\left((\mathcal{\widetilde{\nabla}}v) u+v\mathcal{L}u\right) d\nu=\int_{\partial\Omega}v\langle \mathcal{\widetilde{\nabla} }u,d\nu\rangle,
\end{equation}
where $$\widetilde{\nabla} u=\sum_{k=1}^{N_{1}} (X_{k} u) X_{k},$$
see Proposition \ref{green1}. Consequently, we
apply it to give simple proofs of the existence and uniqueness
for some boundary value problems of Dirichlet, von Neumann, mixed Dirichlet-Neumann, and Robin types for the sub-Laplacian, as well for the sub-Laplacian (stationary)
Schr\"odinger operator,
see Section \ref{SEC:2}. These formulations (except for the Dirichlet one) appear to be new. We also apply  \eqref{g10} to obtain a refined ``local" Hardy inequality in Section \ref{Sec7}.
The following analogue of Green's second formula is also established:
if $u,v\in C^{2}(\Omega)\bigcap C^{1}(\overline{\Omega}),$ then
\begin{equation}\label{g20}
\int_{\Omega}(u\mathcal{L}v-v\mathcal{L}u)d\nu
=\int_{\partial\Omega}(u\langle\widetilde{\nabla}  v,d\nu\rangle-v\langle \widetilde{\nabla}  u,d\nu\rangle),
\end{equation}
see Proposition \ref{green2}. As a consequence we obtain several
representation formulae for functions in $\Omega$.

We note that up to here the obtained formulae can be also formulated in terms of
the perimeter measure, see e.g. \cite{CGH08}. We outline a relation between these two integrations in Section \ref{SEC:mes}. However, the advantage of using the language of differential forms is in the possibility of making coordinate free formulations which will prove to be effective in the subsequent applications.

\item
We discuss the single layer potentials (for the sub-Laplacian $\mathcal{L}$) in the form
\begin{equation*}\label{S0}
\mathcal{S}_{j} u(x)=\int_{\partial\Omega} u(y) \varepsilon(x,y) \langle X_j, d\nu(y)\rangle,\quad j=1,...,N_{1},
\end{equation*}
where $$\varepsilon(x,y)=\varepsilon(y,x)=\varepsilon(x^{-1}y)$$ is the
fundamental solution of the sub-Laplacian $\mathcal{L}$ on $\mathbb G$,
$$\mathcal{L}\varepsilon=\delta,$$
see \eqref{fundsol}
for its formula.
The advantage of this definition
is that it becomes integrable over the whole boundary including
also the characteristic points (in comparison, for example, to the one used by Jerison
\cite{J} which is not integrable over characteristic points).
This becomes very useful for subsequent analysis. We show that if $u\in L^{\infty}(\partial\Omega)$ then
$\mathcal{S}_{j}u$ is continuous (Theorem
\ref{S}).
As the double layer potential we consider
\begin{equation*}\label{EQ:dp0}
\mathcal{D}u(x)=\int_{\partial \Omega} u(y)\langle \widetilde{\nabla} \varepsilon(x,y),d\nu(y)\rangle,
\end{equation*}
where $$\widetilde{\nabla} \varepsilon=\sum_{k=1}^{N_{1}} (X_{k}\varepsilon) X_{k},$$
and we establish its jump relations in Theorem \ref{doublelayer}.
We use these potentials extensively for further analysis.

\item
We establish trace formulae for the Newton potential operator
$$\mathcal N f(x)=\int_{\mathbb G}
\varepsilon(x,y) f(y) d\nu(y)$$
to arbitrary bounded piecewise smooth surfaces $\partial \Omega$ when $\text{supp} f\subset \Omega\subset \mathbb G$ and $f$ is in the Folland-Stein's H\"older space.
We then use this to introduce a version of Kac's boundary value problem on homogeneous Carnot groups and Kac's principle of ``not feeling the boundary'' for the sub-Laplacian $\mathcal{L}$,
see Section \ref{SEC:5}.

\item
We carry out the above analysis of traces and Kac's problem also for poly-sub-Laplacians
$\mathcal{L}^{m}$, see Section \ref{SEC:6}, for all integers $m\geq 1$.

\item In Section \ref{Sec7} we give another example of using the techniques from this paper, in particular of Green's first formula, to obtain a ``local" Hardy inequality on $\mathbb G$.
Namely, we show that for $\alpha\in \mathbb{R}$, $\alpha>2-Q$ and $Q\geq 3$, we have
\begin{multline*}\label{LH20}
\qquad \qquad \int_{\Omega}d^{\alpha} |\nabla_{\mathbb G} u|^{2} \,d\nu\geq
\left(\frac{\alpha+Q-2}{2}\right)^{2}\int_{\Omega}
d^{\alpha}\frac{|\nabla_{\mathbb G} d|^{2}}{d^{2}}
|u|^{2}\,d\nu\\+
\frac{\alpha+Q-2}{2(Q-2)}\int_{\partial\Omega}d^{\alpha+Q-2}|u|^{2}
\langle\widetilde{\nabla}d^{2-Q},
d\nu\rangle,
\end{multline*}
for all $u\in C^{1}(\Omega)\bigcap C(\overline{\Omega})$, where $d$ is the
the $\mathcal{L}$-gauge distance and $$\nabla_{\mathbb
G}=(X_{1},\ldots,X_{N_{1}}).$$ If $u=0$ on $\partial\Omega$, the second
(boundary) term vanishes, and the constant
$\left(\frac{Q+\alpha-2}{2}\right)^{2}$ in the first term is known to be sharp.
Since this last (boundary) term can be $\geq 0$, this provides a refinement to
the known Hardy inequality. Consequently, we also obtain refined versions of
(local) uncertainty principles on $\mathbb G$.

For other types of the Hardy and other inequalities on Carnot groups as well as for a literature review on this subject we refer to \cite{Ruzhansky-Suragan:JDE}, as well as to \cite{Ruzhansky-Suragan:equalities} for the setting of general homogeneous groups.

\item We discuss how all the results can be extended to operators
$$\mathcal{L}_{A}=\sum_{k,j=1}^{N_{1}}a_{k,j}X_{k}X_{j},$$
where $A=(a_{k,j})_{1\leq k,j\leq N_{1}}$ is a positive-definite symmetric matrix, at least in
the setting of free homogeneous Carnot groups; see Section \ref{SEC:prelim}.

\end{itemize}

In Section \ref{SEC:prelim} we very briefly review the main concepts of homogeneous Carnot groups and fix the notation.
 In Section \ref{SEC:2} we derive versions of Green's first and second formulae,
and give applications to boundary value problems of different types. Boundary layer potentials for the sub-Laplacian on homogeneous Carnot groups are presented and analysed in Section \ref{SEC:3}.
We derive trace formulae and give the analogues of Kac's boundary value problem for the sub-Laplacian and higher powers of the sub-Laplacian in Section \ref{SEC:5} and \ref{SEC:6}, respectively. In Section \ref{Sec7} we make a short discussion of a ``local'' Hardy inequality and a ``local'' uncertainty principle.

\medskip

The authors would like to thank Nicola Garofalo and Valentino Magnani for enlightening discussions.

\section{Preliminaries}
\label{SEC:prelim}

There are several equivalent definitions of homogeneous Carnot groups.
We follow the definition in \cite{BLU07} (but see also e.g. \cite{FR14, FR} for the Lie algebra point of view):

\begin{defn} \label{maindef}
A Lie group $\mathbb{G}=(\mathbb{R}^{N},\circ)$ is called homogeneous Carnot group
(or homogeneous stratified group) if it satisfies the following conditions:

(a) For some natural numbers $N_{1}+...+N_{r}=N$
the decomposition $\mathbb{R}^{N}=\mathbb{R}^{N_{1}}\times...\times\mathbb{R}^{N_{r}}$ is valid, and
for every $\lambda>0$ the dilation $\delta_{\lambda}: \mathbb{R}^{N}\rightarrow \mathbb{R}^{N}$
given by
$$\delta_{\lambda}(x)\equiv\delta_{\lambda}(x^{(1)},...,x^{(r)}):=(\lambda x^{(1)},...,\lambda^{r}x^{(r)})$$
is an automorphism of the group $\mathbb{G}.$ Here $x^{(k)}\in \mathbb{R}^{N_{k}}$ for $k=1,...,r.$

(b) Let $N_{1}$ be as in (a) and let $X_{1},...,X_{N_{1}}$ be the left invariant vector fields on $\mathbb{G}$ such that
$X_{k}(0)=\frac{\partial}{\partial x_{k}}|_{0}$ for $k=1,...,N_{1}.$ Then
$${\rm rank}({\rm Lie}\{X_{1},...,X_{N_{1}}\})=N,$$
for every $x\in\mathbb{R}^{N},$ i.e. the iterated commutators
of $X_{1},...,X_{N_{1}}$ span the Lie algebra of $\mathbb{G}.$
\end{defn}

That is, we say that the triple $\mathbb{G}=(\mathbb{R}^{N},\circ, \delta_{\lambda})$ is a homogeneous Carnot group.
In Definition \ref{maindef}, $r$ is called a step of $\mathbb{G}$ and the left invariant vector fields $X_{1},...,X_{N_{1}}$ are called
the (Jacobian) generators of $\mathbb{G}$. We also note that a Lie group satisfying the condition (a) is called a homogeneous Lie group (modelled on $\mathbb{R}^{N}$ but this is not restricting the generality).
The number
$$Q=\sum_{k=1}^{r}kN_{k}.$$
is called the homogeneous dimension of $G$.

\smallskip
Throughout this paper we assume $Q\geq 3$. This is not very restrictive since it
effectively rules out only the spaces $\mathbb R$ and $\mathbb R^{2}$ where the
fundamental solution assumes a different form and where most things are already known.

\smallskip
The second order differential operator
\begin{equation}\label{sublap}
\mathcal{L}=\sum_{k=1}^{N_{1}}X_{k}^{2}
\end{equation}
is called the (canonical) sub-Laplacian on $\mathbb{G}$.
In this paper we mainly consider the operator \eqref{sublap}.
The sub-Laplacian $\mathcal{L}$ is a left invariant homogeneous hypoelliptic differential operator and it is known that $\mathcal{L}$ is elliptic if and only if the step of $\mathbb{G}$ is equal to 1.

In general, many of our results can be extended to any second order hypoelliptic differential operators which are ``equivalent" to the sub-Laplacian $\mathcal{L}$. Let us very briefly discuss this matter in the sprit of \cite{BLU07}.

Let $A=(a_{k,j})_{1\leq k,j\leq N_{1}}$ be a positive-definite symmetric matrix. Consider the following second order hypoelliptic differential operator based on the matrix $A$ and the vector fields $\{X_{1},...,X_{N_{1}}\}$, given by
\begin{equation}
\mathcal{L}_{A}=\sum_{k,j=1}^{N_{1}}a_{k,j}X_{k}X_{j}.
\end{equation}
For instance, in the Euclidean case ($N_{1}=N$), that is, for $\mathbb{G}=(\mathbb{R}^{N},+),$
the constant coefficient second order operator of elliptic type
$$
\Delta_{A}=
\sum_{k,j=1}^{N}a_{k,j}\frac{\partial^{2}}{\partial x_{k}\partial x_{j}},
$$
is transformed into
$$
\Delta=
\sum_{k=1}^{N}\frac{\partial^{2}}{\partial x^{2}_{k}},
$$
under a linear change of coordinates in $\mathbb{R}^{N}.$
Thus, the operator $\Delta_{A}$ is ``equivalent" to the operator
$\Delta$ (in a new system of coordinates).
Therefore, an essential question is whether this idea can be carried out
for $\mathcal{L}_{A}$ on the group $\mathbb{G}$?
In general, classical changes of basis fail to apply, that is, an additional assumption on the group $\mathbb{G}$ is needed in order to preserve the above idea. For example, the idea is preserved if $\mathbb{G}$ is a free Carnot group. Recall that the Carnot group $\mathbb{G}$ is called a free Carnot group if its Lie algebra is (isomorphic to) a free Lie algebra. For instance, the Heisenberg group $\mathbb{H}^{1}$ is a free Carnot group.
\begin{thm}[\cite{BLU07}]
Let $\mathbb{G}$ be a free homogeneous Carnot group, and let $A$ be a given positive-definite symmetric matrix.
Let $X=\{X_{1},...,X_{N_{1}}\}$ be left-invariant vector fields in the first stratum of the free homogeneous Carnot group $\mathbb{G}$ with the corresponding sub-Laplacian \eqref{sublap}.
Let
\begin{equation}
Y_{k}:=\sum_{j=1}^{N_{1}}\left( A^{\frac{1}{2}}\right)_{k,j}X_{j},\quad k=1,...,N_{1}.
\end{equation}
Consider the related second order differential operator
$$
\mathcal{L}_{A}=
\sum_{k=1}^{N_{1}}Y^{2}_{k}=\sum_{k,j=1}^{N_{1}}a_{k,j}X_{k}X_{j}.
$$
Then there exists a Lie group automorphism $T_{A}$ of $\mathbb{G}$ such that
$$Y_{k}(u\circ T_{A})=(X_{k}u)\circ T_{A},\quad k=1,...,N_{1},$$
$$\mathcal{L}_{A}(u\circ T_{A})=(\mathcal{L}u)\circ T_{A},$$
for every smooth function $u:\mathbb{G}\rightarrow \mathbb{R}.$ Moreover, $T_{A}$ has polynomial component
functions and commutes with the dilations of $\mathbb{G}$.
\end{thm}

It can be also proved that if $\mathbb{G}$ is not a free Carnot group, then the automorphism $T_{A}$
may not exist. However, it can be shown that for any homogeneous Carnot group $\mathbb{G}$ there exists a different homogeneous Carnot group $\mathbb{G}_{\ast}=(\mathbb{R}^{N}, \ast,\delta_{\lambda})$, that is, with the
same underlying manifold $\mathbb{R}^{N}$ and the same group of dilations $\delta_{\lambda}$ as $\mathbb{G},$ and a Lie-group isomorphism from
$\mathbb{G}$ to $\mathbb{G}_{\ast}$ turning $\mathcal{L}_{A}$ (of $\mathbb{G}$) into $\mathcal{L}$ (of $\mathbb{G}_{1}$) (see \cite[Chapter 16.3]{BLU07}).
See also \cite{ruzhansky+turunen-book} for analogues of such constructions on compact Lie groups.

\section{Sub-Laplacian Green's formulae and their consequences}
\label{SEC:2}

In this section we will give some important lemmas for integration on $\mathbb{G}$ and their consequences for a number of boundary value problems of different types. These results will be also used in the other sections.

It is known that the sub-Laplacian has a unique fundamental solution
$\varepsilon$ on $\mathbb{G}$ (see \cite{Fol75}),
$$
\mathcal{L}\varepsilon=\delta,
$$
and $\varepsilon(x,y)=\varepsilon(y^{-1}x)$ is homogeneous of degree $-Q+2$ and represented in the form
\begin{equation}\label{fundsol}
\varepsilon(x,y)= [d(x,y)]^{2-Q},
\end{equation}
where $d$ is the $\mathcal{L}$-gauge.

We recall that the $\mathcal{L}$-gauge (see \cite{BLU07}) is a symmetric homogeneous (quasi-) norm on the homogeneous Carnot
group $\mathbb{G}=(\mathbb{R}^{N},\circ, \delta_{\lambda})$, that is,
\begin{itemize}
\item $d(x)>0$ if and only if $x\neq 0,$
\item $d(\delta_{\lambda}(x))=\lambda d(x)$ for all $\lambda>0$ and $x\in \mathbb{G}$,
\item $d(x^{-1})=d(x)$ for all $x\in\mathbb{G}$.
\end{itemize}
Throughout this paper we keep the notation $d$ for the $\mathcal{L}$-gauge and use the following quasi-metric properties:
\begin{itemize}
\item $d(x,y)\geq 0$, and $d(x,y)=0$ if and only if $x=y$ for all $x,y\in \mathbb{G}$.
\item There exists a positive constant $C\geq 1$ such that
$$d(x,y)\leq C(d(x,z)+d(z,y))$$
 for all $x,y,z\in \mathbb{G}$.
\end{itemize}

Let $Q\geq 3$ be the homogeneous dimension of $\mathbb{G}$, $\partial\Omega$ the boundary of an admissible domain $\Omega$ in $\mathbb{G}$, $d\nu$ the volume element on $\mathbb{G}$, $X=\{X_{1},...,X_{N_{1}}\}$ left-invariant vector fields in the first stratum of a homogeneous Carnot group $\mathbb{G}$, and $\langle X_{j}, d\nu\rangle$ the natural pairing between vector fields and differential forms,
see \eqref{Xk}, \eqref{theta} and \eqref{Xjdnu} for the derivation of the exact formula: more precisely,
as we will see in the proof of Proposition \ref{stokes}, we can write
\begin{equation}\label{Xjdnu0}
\langle X_{k}, d\nu(x)\rangle=\bigwedge_{j=1,j\neq k}^{N_{1}}dx_{j}^{(1)}\bigwedge_{l=2}^{r}
\bigwedge_{m=1}^{N_{l}}\theta_{l,m},
\end{equation}
with
\begin{equation}\label{theta0}
\theta_{l,m}=-\sum_{k=1}^{N_{1}}a_{k,m}^{(l)}(x^{(1)},\ldots,x^{(l-1)})
 dx_{k}^{(1)}+dx_{m}^{(l)},\,\,l=2,\ldots,r,\,\,m=1,\ldots,N_{l},
\end{equation}
where $a_{k,m}^{(l)}$ is a $\delta_{\lambda}$-homogeneous polynomial of degree $l-1$
such that
\begin{equation}\label{Xk0}
X_{k}=\frac{\partial}{\partial x_{k}^{(1)}}+
\sum_{l=2}^{r}\sum_{m=1}^{N_{l}}a_{k,m}^{(l)}(x^{(1)},...,x^{(l-1)})
\frac{\partial}{\partial x_{m}^{(l)}}.
\end{equation}
We also recall that the standard Lebesque measure on $\mathbb R^{N}$ is the Haar measure for $\mathbb{G}$ (see, e.g. \cite[Proposition 1.3.21]{BLU07}).
First we prove some important properties for integration. The following proposition can be also obtained using an abstract Cartan formula, see Section \ref{SEC:mes}, but here we give a direct explicit proof in order to derive an explicit formula for the forms $\theta_{l,m}$ in \eqref{theta0}.

The advantage of using the language of differential forms here and in subsequent Green's formulae (e.g. over the perimeter or the surface measure) is in the possibility of making coordinate free formulations which will prove to be effective in the proof of the Hardy inequality in the sequel.

\begin{prop}[Divergence formula]
\label{stokes}
Let $\Omega\subset\mathbb{G}$ be an admissible domain. Let $f_{k}\in C^{1}(\Omega)\bigcap C(\overline{\Omega}),\,k=1,\ldots,N_{1}$. Then for each $k=1,\ldots,N_{1},$ we have
\begin{equation}\label{EQ:S1}
\int_{\Omega}X_{k}f_{k}d\nu=
\int_{\partial\Omega}f_{k} \langle X_{k},d\nu\rangle.
\end{equation}
Consequently, we also have
\begin{equation}\label{EQ:S2}
\int_{\Omega}\sum_{k=1}^{N_{1}}X_{k}f_{k}d\nu=
\int_{\partial\Omega}\sum_{k=1}^{N_{1}}f_{k} \langle X_{k},d\nu\rangle.
\end{equation}
\end{prop}

\begin{proof}[Proof of Proposition \ref{stokes}]
Recall
(see e.g. \cite[Section 3.1.5]{FR})
that the Jacobian basis $\{X_{1},...,X_{N_{1}}\}$ of $\mathbb{G}$ takes the form
\begin{equation}\label{Xk}
X_{k}=\frac{\partial}{\partial x_{k}^{(1)}}+
\sum_{l=2}^{r}\sum_{m=1}^{N_{l}}a_{k,m}^{(l)}(x^{(1)},...,x^{(l-1)})
\frac{\partial}{\partial x_{m}^{(l)}},
\end{equation}
where $a_{k,m}^{(l)}$ is a $\delta_{\lambda}$-homogeneous polynomial function of degree $l-1.$
As in Definition \ref{maindef}, $r$ is the step of $\mathbb{G}$ and
$x^{(l)}\in \mathbb{R}^{N_{l}}$, $l=1,\ldots,r$.
For any function $f$ we obtain the following differentiation formula

$$df=\sum_{k=1}^{N_{1}}\frac{\partial f}{\partial x_{k}^{(1)}}dx_{k}^{(1)}+\sum_{l=2}^{r}\sum_{m=1}^{N_{l}}
\frac{\partial f}{\partial x_{m}^{(l)}} dx_{m}^{(l)}
$$
$$
=\sum_{k=1}^{N_{1}} X_{k}f dx_{k}^{(1)}-\sum_{k=1}^{N_{1}}\sum_{l=2}^{r}
\sum_{m=1}^{N_{l}}a_{k,m}^{(l)}(x^{(1)},\ldots,x^{(l-1)})
\frac{\partial f}{\partial x_{m}^{(l)}} dx_{k}^{(l)}$$
$$+
\sum_{l=2}^{r}\sum_{m=1}^{N_{l}}
\frac{\partial f}{\partial x_{m}^{(l)}} dx_{m}^{(l)
}=\sum_{k=1}^{N_{1}} X_{k}f dx_{k}^{(1)}$$
$$+\sum_{l=2}^{r}\sum_{m=1}^{N_{l}}\frac{\partial f}
{\partial x_{m}^{(l)}} (-\sum_{k=1}^{N_{1}}a_{k,m}^{(l)}(x^{(1)},\ldots,x^{(l-1)})
 dx_{k}^{(1)}+dx_{m}^{(l)})$$
$$
=\sum_{k=1}^{N_{1}} X_{k}f dx_{k}^{(1)}+\sum_{l=2}^{r}\sum_{m=1}^{N_{l}}\frac{\partial f}
{\partial x_{m}^{(l)}} \theta_{l,m},
$$
where
\begin{equation}\label{theta}
\theta_{l,m}=-\sum_{k=1}^{N_{1}}a_{k,m}^{(l)}(x^{(1)},\ldots,x^{(l-1)})
 dx_{k}^{(1)}+dx_{m}^{(l)},\,\,l=2,\ldots,r,\,\,m=1,\ldots,N_{l}.
\end{equation}
That is
\begin{equation}\label{df}
df=\sum_{k=1}^{N_{1}} X_{k}f dx_{k}^{(1)}+\sum_{l=2}^{r}\sum_{m=1}^{N_{l}}\frac{\partial f}
{\partial x_{m}^{(l)}} \theta_{l,m}.
\end{equation}
It is simple to see that
 $$\langle X_{s}, dx_{j}^{(1)}\rangle=
 \frac{\partial}{\partial x_{s}^{(1)}}dx_{j}^{(1)}=\delta_{sj},$$
where $\delta_{sj}$ is the Kronecker delta, and
$$\langle X_{s}, \theta_{l,m} \rangle=
\left(\frac{\partial}{\partial x_{s}^{(1)}}+
\sum_{h=2}^{r}\sum_{g=1}^{N_{h}}a_{s,g}^{(h)}(x^{(1)},...,x^{(h-1)})
\frac{\partial}{\partial x_{g}^{(h)}}\right)
$$
$$
\cdot\left(
-\sum_{k=1}^{N_{1}}a_{k,m}^{(l)}(x^{(1)},\ldots,x^{(l-1)})
 dx_{k}^{(1)}+dx_{m}^{(l)}
\right)
$$
$$=-\sum_{k=1}^{N_{1}}\left(\frac{\partial}{\partial x_{s}^{(1)}}a_{k,m}^{(l)}(x^{(1)},\ldots,x^{(l-1)})\right)
 dx_{k}^{(1)}$$
$$-\sum_{k=1}^{N_{1}}a_{k,m}^{(l)}(x^{(1)},\ldots,x^{(l-1)})
\frac{\partial}{\partial x_{s}^{(1)}}
 dx_{k}^{(1)}
+\frac{\partial}{\partial x_{s}^{(1)}}dx_{m}^{(l)}
$$
 $$
 -\sum_{k=1}^{N_{1}}\sum_{h=2}^{r}\sum_{g=1}^{N_{h}}a_{s,g}^{(h)}
 (x^{(1)},...,x^{(h-1)})\left(\frac{\partial}{\partial x_{g}^{(h)}}a_{k,m}^{(l)}(x^{(1)},\ldots,x^{(l-1)})\right)dx_{k}^{(1)}
 $$
 $$
 -\sum_{k=1}^{N_{1}}\sum_{h=2}^{r}\sum_{g=1}^{N_{h}}a_{s,g}^{(h)}
 (x^{(1)},...,x^{(h-1)})a_{k,m}^{(l)}(x^{(1)},\ldots,x^{(l-1)})\frac{\partial}{\partial x_{g}^{(h)}}dx_{k}^{(1)}
 $$
 $$+
\sum_{h=2}^{r}\sum_{g=1}^{N_{h}}a_{s,g}^{(h)}(x^{(1)},...,x^{(h-1)})
\frac{\partial}{\partial x_{g}^{(h)}}dx_{m}^{(l)}$$

$$=-\sum_{k=1}^{N_{1}}\left(\frac{\partial}{\partial x_{s}^{(1)}}a_{k,m}^{(l)}(x^{(1)},\ldots,x^{(l-1)})\right)
 dx_{k}^{(1)}
-\sum_{k=1}^{N_{1}}a_{k,m}^{(l)}(x^{(1)},\ldots,x^{(l-1)})\delta_{sk}$$
 $$
 -\sum_{k=1}^{N_{1}}\sum_{h=2}^{r}\sum_{g=1}^{N_{h}}a_{s,g}^{(h)}
 (x^{(1)},...,x^{(h-1)})\left(\frac{\partial}{\partial x_{g}^{(h)}}a_{k,m}^{(l)}(x^{(1)},\ldots,x^{(l-1)})\right)dx_{k}^{(1)}
 $$
$$+
\sum_{h=2}^{r}\sum_{g=1}^{N_{h}}a_{s,g}^{(h)}(x^{(1)},...,x^{(h-1)})
\delta_{gm}\delta_{hl}=$$
 $$
 -\sum_{k=1}^{N_{1}}\sum_{h=2}^{r}\sum_{g=1}^{N_{h}}a_{s,g}^{(h)}
 (x^{(1)},...,x^{(h-1)})\left(\frac{\partial}{\partial x_{g}^{(h)}}a_{k,m}^{(l)}(x^{(1)},\ldots,x^{(l-1)})\right)dx_{k}^{(1)}
 $$
$$
-\sum_{k=1}^{N_{1}}\left(\frac{\partial}{\partial x_{s}^{(1)}}a_{k,m}^{(l)}(x^{(1)},\ldots,x^{(l-1)})\right)
 dx_{k}^{(1)}=$$
$$
-\sum_{k=1}^{N_{1}} \bigg[\sum_{h=2}^{r}\sum_{g=1}^{N_{l}}a_{s,g}^{(h)}
 (x^{(1)},...,x^{(h-1)})\left(\frac{\partial}{\partial x_{g}^{(h)}}a_{k,m}^{(l)}(x^{(1)},\ldots,x^{(l-1)})\right)
 $$
 $$
 +\frac{\partial}{\partial x_{s}^{(1)}}a_{k,m}^{(l)}(x^{(1)},\ldots,x^{(l-1)})\bigg]
 dx_{k}^{(1)}.
$$
That is, we have
 $$\langle X_{s}, dx_{j}^{(1)}\rangle=\delta_{sj},$$
for $s,j=1,\ldots,N_{1},$ and
$$\langle X_{s}, \theta_{l,m} \rangle=
\sum_{k=1}^{N_{1}}\mathcal{C}_{k} dx_{k}^{(1)},
$$
for $s=1,\ldots,N_{1},\,\,l=2,\ldots,r,\,\,m=1,\ldots,N_{l}$.
Here $$\mathcal{C}_{k}=-\sum_{h=2}^{r}\sum_{g=1}^{N_{l}}a_{s,g}^{(h)}
 (x^{(1)},...,x^{(h-1)})\frac{\partial}{\partial x_{g}^{(h)}}a_{k,m}^{(l)}(x^{(1)},\ldots,x^{(l-1)})
 $$
 $$-\frac{\partial}{\partial x_{s}^{(1)}}a_{k,m}^{(l)}(x^{(1)},\ldots,x^{(l-1)}).$$
We have
$$d\nu:=d\nu(x)=\bigwedge_{j=1}^{N}dx_{j}
=
\bigwedge_{j=1}^{N_{1}}dx_{j}^{(1)}\bigwedge_{l=2}^{r}\bigwedge_{m=1}^{N_{l}}dx_{m}^{(l)}
=\bigwedge_{j=1}^{N_{1}}dx_{j}^{(1)}\bigwedge_{l=2}^{r}\bigwedge_{m=1}^{N_{l}}\theta_{l,m},$$
so
\begin{equation}\label{Xjdnu}
\langle X_{k}, d\nu(x)\rangle=\bigwedge_{j=1,j\neq k}^{N_{1}}dx_{j}^{(1)}\bigwedge_{l=2}^{r}
\bigwedge_{m=1}^{N_{l}}\theta_{l,m}.
\end{equation}
Therefore, by using Formula \eqref{df} we get
$$d(f_{s}\langle X_{s}, d\nu(x)\rangle)=df_{s}\wedge\langle X_{s}, d\nu(x)\rangle
$$$$=\sum_{k=1}^{N_{1}} X_{k}f_{s} dx_{k}^{(1)}\wedge\langle X_{s}, d\nu(x)\rangle$$$$+\sum_{l=2}^{r}\sum_{m=1}^{N_{l}}\frac{\partial f_{s}}
{\partial x_{m}^{(l)}} \theta_{l,m}\wedge\langle X_{s}, d\nu(x)\rangle=X_{s}f_{s}d\nu(x),$$
that is,
\begin{equation}\label{stokesone}
d(\langle f_{k}X_{k}, d\nu(x)\rangle)=X_{k}f_{k}d\nu(x),\quad k=1,\ldots,N_{1}.
\end{equation}
Now using the Stokes theorem (see e.g. \cite[Theorem 26.3.1]{DFN}) we obtain \eqref{EQ:S1}.
Taking a sum over $k$ we also obtain \eqref{EQ:S2}.
\end{proof}

We have the following analogue of Green's first formula. This version was proved for the ball in \cite{Gav77} and for any smooth domain of the complex Heisenberg group in \cite{R91}. See also \cite{BLU07} and \cite{CGH08} for other analogues.
\begin{prop}[Green's first formula]\label{green1}
Let $\Omega\subset\mathbb{G}$ be an admissible domain. Let $v\in C^{1}(\Omega)\bigcap C(\overline{\Omega})$ and $u\in C^{2}(\Omega)\bigcap C^{1}(\overline{\Omega})$. Then
\begin{equation} \label{g1}
\int_{\Omega}\left((\mathcal{\widetilde{\nabla}}v) u+v\mathcal{L}u\right) d\nu=\int_{\partial\Omega}v\langle \mathcal{\widetilde{\nabla} }u,d\nu\rangle,
\end{equation}
where $\mathcal{L}$ is the sub-Laplacian on $\mathbb{G}$ and
$$\mathcal{\widetilde{\nabla} }u=\sum_{k=1}^{N_{1}}\left(X_{k}u\right)X_{k}.$$
\end{prop}

\begin{proof}[Proof of Proposition \ref{green1}]
Let $f_{k}=vX_{k}u,$ then

$$\sum_{k=1}^{N_{1}}X_{k}f_{k}=(\mathcal{\widetilde{\nabla} }v) u+v\mathcal{L}u.$$
Here as usual we understand the scalar expression for
$(\mathcal{\widetilde{\nabla} }v) u$ as
$$\left(\mathcal{\widetilde{\nabla} }v\right) u=\mathcal{\widetilde{\nabla} }v u= \sum_{k=1}^{N_{1}}\left(X_{k}v\right)\left(X_{k}u\right)=\sum_{k=1}^{N_{1}}
X_{k}v X_{k}u.$$
Otherwise, we may sometimes use the expression $\mathcal{\widetilde{\nabla} }(v u),$ of course, this is an operator.
By using Divergence formula in Proposition \ref{stokes} we obtain
\begin{multline*}
\int_{\Omega}\left(\mathcal{\widetilde{\nabla} }v u+v\mathcal{L}u\right) d\nu$$$$=\int_{\Omega}
\sum_{k=1}^{N_{1}}X_{k}f_{k}d\nu \\
=\int_{\partial\Omega}\sum_{k=1}^{N_{1}}\langle f_{k}X_{k},d\nu\rangle
=\int_{\partial\Omega}\sum_{k=1}^{N_{1}}\langle vX_{k}uX_{k},d\nu\rangle
=\int_{\partial\Omega}v\langle \mathcal{\widetilde{\nabla} }u,d\nu\rangle,
\end{multline*}
completing the proof.
\end{proof}
When $v=1$ Proposition \ref{green1} implies the following analogue of Gauss' mean value formula for
harmonic functions:
\begin{cor}\label{COR:v0}
If $\mathcal{L}u=0$ in an admissible domain $\Omega\subset\mathbb{G}$, then
$$\int_{\partial\Omega}\langle \mathcal{\widetilde{\nabla} }u,d\nu\rangle=0.$$
\end{cor}

As in the classical potential theory, the Green formulae are still valid for functions with (weak) singularities provided we can approximate them by smooth functions. In this sense, without further justification we can apply these Green's formulae, in particular, to the fundamental solution $\varepsilon$.

Then, for $x\in\Omega$, taking $v=1$ and $u(y)=\varepsilon(x,y)$ we record
the following consequence of Proposition \ref{green1}:
\begin{cor}\label{COR:v1}
If $\Omega\subset\mathbb{G}$ is an admissible domain, and $x\in\Omega$, then
$$\int_{\partial\Omega}\langle \mathcal{\widetilde{\nabla} }\varepsilon(x,y),d\nu(y)\rangle=1,$$
where $\varepsilon$ is the fundamental solution of the sub-Laplacian.
\end{cor}

Now we can prove the following uniqueness result by a classical potential theory method. We should mention that
the following lemma is known and can be proved by other methods too, but given
Green's first formula in Proposition \ref{green1} its proof becomes elementary.

In addition to the more well-known Dirichlet boundary conditions we give examples of boundary conditions of different types, such as Neumann, Robin, mixed Dirichlet and Robin, or different types of conditions on different parts of the boundary. For brevity, we restrict the consideration below to zero boundary conditions only, otherwise these problems may become very delicate due to the presence of characteristic points, see e.g. \cite{DGN}. We hope to address these issues with our methods in a subsequent paper; overall, this is a very active area of research, see e.g.
\cite{BG} and references therein also for other types of equations.

We also note that in the subsequently considered boundary value problems, we can assume
without loss of generality (in the proofs) that functions are real valued since otherwise we can always take real and imaginary parts which would then satisfy the same equations.
As usual, throughout this paper $\Omega$ is an admissible domain, see Definition \ref{DEF:adomain}.

\begin{lem}\label{uniqueness}
The Dirichlet boundary value problem
\begin{equation}\label{Lu0}
\mathcal{L}u(x)=0,\,\, x\in\Omega\subset\mathbb{G},
\end{equation}
\begin{equation}\label{u0}
u(x)=0,\,\, x\in\partial\Omega,
\end{equation}
has the unique trivial solution
$u\equiv0$ in the class of functions $C^{2}(\Omega)\bigcap C^{1}(\overline{\Omega})$.
\end{lem}

\begin{proof}[Proof of Lemma \ref{uniqueness}.]
Set $v=u$ in \eqref{g1}, then by \eqref{Lu0} and \eqref{u0} we get
$$\int_{\Omega}\mathcal{\widetilde{\nabla} }u u d\nu=\int_{\Omega}\left(\mathcal{\widetilde{\nabla} }u u+u\mathcal{L}u\right)d\nu
=\int_{\partial\Omega}u\langle \mathcal{\widetilde{\nabla} }u,d\nu\rangle=0.$$
Therefore
$$\int_{\Omega}\sum_{k=1}^{N_{1}}|X_{k}u|^{2}d\nu=0,$$
that is,
$X_{k}u=0,\,\, k=1,...,N_{1}.$ Since any element of a Jacobian basis of $\mathbb{G}$ is represented by Lie brackets of $\{X_{1},...,X_{N_{1}}\}$, we obtain that
$u$ is a constant, so $u\equiv0$ on $\Omega$ by \eqref{u0}.
\end{proof}

This has the following simple extension to (stationary) Schr\"odinger operators:

\begin{lem}\label{schr}
Let $q:\Omega\rightarrow \mathbb{R}$ be a non-negative bounded function that is,
$q\in L^{\infty}(\Omega)$ and $q(x)\geq 0,\, x\in \Omega$.
Then the Dirichlet boundary value problem for the Schr\"odinger equation
\begin{equation}\label{Su0}
-\mathcal{L}u(x)+q(x)u(x)=0,\,\, x\in\Omega\subset\mathbb{G},
\end{equation}
\begin{equation}\label{s0}
u(x)=0,\,\, x\in\partial\Omega,
\end{equation}
has the unique trivial solution
$u\equiv0$ in the class of functions $C^{2}(\Omega)\bigcap C^{1}(\overline{\Omega})$.
\end{lem}

\begin{proof}[Proof of Lemma \ref{schr}.]
As in proof of Lemma  \ref{uniqueness} using Green's formula, from \eqref{Su0} and \eqref{s0} we obtain
\begin{multline*}
\int_{\Omega}\mathcal{\widetilde{\nabla} }u u d\nu=\int_{\Omega}\left(\mathcal{\widetilde{\nabla} }u u
+u\mathcal{L}u\right)d\nu
-\int_{\Omega}q(y)|u(y)|^{2}d\nu \\
=\int_{\partial\Omega}u\langle \mathcal{\widetilde{\nabla} }u,d\nu\rangle
-\int_{\Omega}q(y)|u(y)|^{2}d\nu=
-\int_{\Omega}q(y)|u(y)|^{2}d\nu.
\end{multline*}
Therefore,
$$0\leq \int_{\Omega}\sum_{k=1}^{N_{1}}|X_{k}u|^{2}d\nu=-\int_{\Omega}q(y)|u(y)|^{2}d\nu\leq 0,$$
that is, $u\equiv0$.
\end{proof}

Similarly, we obtain the following statement for the von Neumann type boundary conditions.
We note that von Neumann type boundary value problem
for the sub-Laplacian have been known and studied, see e.g. \cite{DGN}. However, here we offer a new measure-type condition for the von Neumann type boundary value problem for the sub-Laplacian:

\begin{lem}\label{neumann}
The boundary value problem
\begin{equation}\label{Nu0}
\mathcal{L}u(x)=0,\,\, x\in\Omega\subset\mathbb{G},
\end{equation}
\begin{equation}\label{n0}
\sum_{j=1}^{N_{1}}X_{j}u\langle X_{j} ,d\nu\rangle=0\;\textrm{ on }\;\partial\Omega,
\end{equation}
has a solution
$u\equiv \text{const}$ in the class of functions $C^{2}(\Omega)\bigcap C^{1}(\overline{\Omega})$.
\end{lem}

\begin{proof}[Proof of Lemma \ref{neumann}.]
Set $v=u$ in \eqref{g1}, then by \eqref{Nu0} and \eqref{n0} we get
$$\int_{\Omega}\mathcal{\widetilde{\nabla} }u u d\nu=\int_{\Omega}\left(\mathcal{\widetilde{\nabla} }u u+u\mathcal{L}u\right)d\nu
=\int_{\partial\Omega}u\langle \mathcal{\widetilde{\nabla} }u,d\nu\rangle=\int_{\partial\Omega}u\sum_{j=1}^{N_{1}}X_{j}u\langle X_{j} ,d\nu\rangle=0.$$
Therefore
$$\int_{\Omega}\sum_{k=1}^{N_{1}}|X_{k}u|^{2}d\nu=0,$$
that is,
$X_{k}u=0,\,\, k=1,...,N_{1}.$ Since any element of a Jacobian basis of $\mathbb{G}$ is represented by Lie brackets of $\{X_{1},...,X_{N_{1}}\}$, we obtain that
$u$ is a constant.
\end{proof}

Similarly, we can now also consider the Robin type boundary conditions as follows.
\begin{lem}\label{robin}
Let $a_{k}:\partial\Omega\rightarrow {\mathbb R},\,\,k=1,...,N_{1},$ be bounded functions
such that the measure
\begin{equation}\label{m1}
\sum_{j=1}^{N_1}a_j\langle X_{j} ,d\nu\rangle\geq 0
\end{equation}
is non-negative on $\partial\Omega$.
Then the boundary value problem
\begin{equation}\label{Ru0}
\mathcal{L}u(x)=0,\,\, x\in\Omega\subset\mathbb{G},
\end{equation}
\begin{equation}\label{r0}
\sum_{j=1}^{N_{1}}(a_{j}u+X_{j}u)\langle X_{j} ,d\nu\rangle=0\; \textrm{ on }\; \partial\Omega,\,\,
\end{equation}
has a solution
$u\equiv const$ in the class of functions $C^{2}(\Omega)\bigcap C^{1}(\overline{\Omega})$.

\smallskip
Moreover, if the integral of the measure \eqref{m1} is positive, i.e. if
\begin{equation}\label{m2}\int_{\partial\Omega} \sum_{j=1}^{N_1}a_j\langle X_{j} ,d\nu\rangle> 0,\end{equation}
 then
the boundary value problem \eqref{Ru0}-\eqref{r0}
has the unique trivial solution
$u\equiv 0$ in the class of functions $C^{2}(\Omega)\bigcap C^{1}(\overline{\Omega})$.
\end{lem}

\begin{proof}[Proof of Lemma \ref{robin}.]
Set $v=u$ in \eqref{g1}, then by \eqref{Ru0} and \eqref{r0} we get
\begin{multline}\label{Robin1}
\int_{\Omega}\mathcal{\widetilde{\nabla} }u u d\nu=\int_{\Omega}\left(\mathcal{\widetilde{\nabla} }u u+u\mathcal{L}u\right)d\nu
=\int_{\partial\Omega}u\langle \mathcal{\widetilde{\nabla} }u,d\nu\rangle \\
=\int_{\partial\Omega}u\sum_{j=1}^{N_{1}}X_{j}u\langle X_{j} ,d\nu\rangle
=-\int_{\partial\Omega}u^{2}\sum_{j=1}^{N_{1}}a_{j}\langle X_{j} ,d\nu\rangle,
\end{multline}
that is,
$$\int_{\Omega}\sum_{k=1}^{N_{1}}|X_{k}u|^{2}d\nu=-\int_{\partial\Omega}u^{2}\sum_{j=1}^{N_{1}}a_{j}\langle X_{j} ,d\nu\rangle.$$
Therefore
$$\int_{\Omega}\sum_{k=1}^{N_{1}}|X_{k}u|^{2}d\nu=0$$
and
$$\int_{\partial\Omega}u^{2}\sum_{j=1}^{N_{1}}a_{j}\langle X_{j} ,d\nu\rangle=0.$$
As above the first equality implies that $u$ is a constant. This proves the first part of the claim.

On the other hand, by the assumption \eqref{m2} the second equality implies $u=0$ on $\partial\Omega$, this means $u\equiv0$ on $\Omega$.
\end{proof}

We can also consider problems where Dirichlet or Robin conditions are imposed on different parts of the boundary:

\begin{cor}\label{DN}
Let $a_{k}:\partial\Omega\rightarrow {\mathbb R},\,\,k=1,...,N_{1},$ be bounded functions
such that the measure
\begin{equation}\label{m1a}
\sum_{j=1}^{N_1}a_j\langle X_{j} ,d\nu\rangle\geq 0
\end{equation}
is non-negative on $\partial\Omega$.
Let $\partial\Omega_{1}\subset \partial\Omega$, $\partial\Omega_{1}\neq \{\emptyset\}$ and  $\partial\Omega_{2}:= \partial\Omega\backslash\partial\Omega_{1}$.
Then the boundary value problem
\begin{equation}\label{DNu0}
\mathcal{L}u(x)=0,\,\, x\in\Omega\subset\mathbb{G},
\end{equation}
\begin{equation}\label{dn0}
u=0 \;\textrm{ on }\; \partial\Omega_{1},
\end{equation}
\begin{equation}\label{dn01}
\sum_{j=1}^{N_{1}}(a_{j}u+X_{j}u)\langle X_{j} ,d\nu\rangle=0 \;\textrm{ on }\; \partial\Omega_{2},\,\,
\end{equation}
has the unique trivial solution
$u\equiv 0$ in the class of functions $C^{2}(\Omega)\bigcap C^{1}(\overline{\Omega})$.
\end{cor}
The proof of Corollary \ref{DN} follows the same argument as that in the proof
of Lemma \ref{robin}, and we observe that the last equality in \eqref{Robin1}
is still valid on both parts $\partial\Omega_{1}$ and $\partial\Omega_{2}$ of the boundary
$\partial\Omega$ using conditions \eqref{dn0} and \eqref{dn01}, respectively.

\medskip
As a consequence of the Green's first formula \eqref{g1} we obtain the following analogue of Green's second formula:
\begin{prop}[Green's second formula]
\label{green2}
Let $\Omega\subset\mathbb{G}$ be an admissible domain. Let $u,v\in C^{2}(\Omega)\bigcap C^{1}(\overline{\Omega}).$ Then
\begin{equation}\label{g2}
\int_{\Omega}(u\mathcal{L}v-v\mathcal{L}u)d\nu
=\int_{\partial\Omega}(u\langle\widetilde{\nabla}  v,d\nu\rangle-v\langle \widetilde{\nabla}  u,d\nu\rangle).
\end{equation}
\end{prop}

\begin{proof}[Proof of Proposition \ref{green2}.]
Rewriting \eqref{g1} we have
$$\int_{\Omega}\left((\mathcal{\widetilde{\nabla} }u) v+u\mathcal{L}v\right)d\nu=\int_{\partial\Omega}u\langle \mathcal{\widetilde{\nabla} }v,d\nu\rangle,$$
$$\int_{\Omega}\left((\mathcal{\widetilde{\nabla} }v) u+v\mathcal{L}u\right)d\nu=\int_{\partial\Omega}v\langle \mathcal{\widetilde{\nabla} }u,d\nu\rangle.$$
By subtracting the second identity from the first one
and using $$(\mathcal{\widetilde{\nabla} }u) v=(\mathcal{\widetilde{\nabla} }v) u$$
we obtain the desired result.
\end{proof}

Putting the fundamental solution $\varepsilon$ instead of $v$ in \eqref{g2} we get the following representation formulae that will be used later but are also of importance on their own.
We list them in the following corollaries.

\begin{cor}\label{repn}
Let $u\in C^{2}(\Omega)\bigcap C^{1}(\overline{\Omega})$. Then for $x\in\Omega$ we have
\begin{multline}\label{rep}
u(x)=\int_{\Omega}\varepsilon(x,y)\mathcal{L}u(y)d\nu(y)\\ +
\int_{\partial\Omega}u(y)
\langle\mathcal{\widetilde{\nabla} }\varepsilon(x,y),d\nu(y)\rangle-
\int_{\partial\Omega}\varepsilon(x,y)\langle\mathcal{\widetilde{\nabla} }u(y),d\nu(y)\rangle.
\end{multline}
\end{cor}

\begin{cor}\label{repnhar}
Let $u\in C^{2}(\Omega)\bigcap C^{1}(\overline{\Omega})$ and $\mathcal{L}u=0$ on $\Omega$, then for $x\in\Omega$ we have
 \begin{equation}\label{rep}
u(x)=\int_{\partial\Omega}u(y)
\langle\mathcal{\widetilde{\nabla} }\varepsilon(x,y),d\nu(y)\rangle-
\int_{\partial\Omega}\varepsilon(x,y)
\langle\mathcal{\widetilde{\nabla}}u(y),d\nu(y)\rangle.
\end{equation}
\end{cor}

\begin{cor}\label{repD}
Let $u\in C^{2}(\Omega)\bigcap C^{1}(\overline{\Omega})$ and
\begin{equation}\label{D1}
u(x)=0,\,\, x\in\partial\Omega,
\end{equation}
then
\begin{equation}
u(x)=\int_{\Omega}\varepsilon(x,y)\mathcal{L}u(y)d\nu(y)-
\int_{\partial\Omega}\varepsilon(x,y)\langle\mathcal{\widetilde{\nabla} }u(y),d\nu(y)\rangle.
\end{equation}
\end{cor}

\begin{cor}\label{repN}
Let $u\in C^{2}(\Omega)\bigcap C^{1}(\overline{\Omega})$ and
\begin{equation}\label{n1}
\sum_{j=1}^{N_{1}}X_{j}u\langle X_{j} ,d\nu\rangle=0 \;\textrm{ on }\; \partial\Omega,
\end{equation}
then
 \begin{equation}\label{rep}
u(x)=\int_{\Omega}\varepsilon(x,y)\mathcal{L}u(y)d\nu(y)+\int_{\partial\Omega}u(y)
\langle\mathcal{\widetilde{\nabla}}\varepsilon(x,y),d\nu(y)\rangle.
\end{equation}
\end{cor}

\section{Single and double layer potentials of the sub-Laplacian}
\label{SEC:3}

Recall that the sub-Laplacian has a unique fundamental solution
$\varepsilon$ on $\mathbb{G}$,
see \eqref{fundsol}, given by
$$
\varepsilon(x,y)= [d(x,y)]^{2-Q},
$$
where $d$ is the $\mathcal{L}$-gauge, so that the function
$\varepsilon$ is homogeneous of degree $-Q+2$.

Let $D\subset \mathbb{R}^{N}$ be an open set with boundary $\partial D$. The set $D$ is called a domain of class $C^{1}$ if for each $x_{0}\in\partial D$ there exist a neighborhood $U_{x_{0}}$ of $x_{0},$ and
a function $\phi_{x_{0}}\in C^{1}(U_{x_{0}}),$ with $|\nabla \phi_{x_{0}}|\geq \alpha>0$ in $U_{x_{0}}$, where $\nabla$ is the standard gradient in $\mathbb{R}^{N}$, such that
$$
D\cap U_{x_{0}}=\{x\in U_{x_{0}}\mid \phi_{x_{0}}(x)<0\},
$$
$$
\partial D\cap U_{x_{0}}=\{x\in U_{x_{0}}\mid \phi_{x_{0}}(x)=0\}.
$$
So let $D$ be an open domain of class $C^{1}$. A point $x_{0}\in \partial D$ is called \emph{characteristic} with respect to fields $\{X_{1},...,X_{N_{1}}\},$ if given $U_{x_{0}},$ $\phi_{x_{0}},$ as above, we have
$$X_{1}\phi_{x_{0}}(x_{0})=0,\ldots, X_{N_{1}}\phi_{x_{0}}(x_{0})=0.$$
Typically, bounded domains have non-empty collection (set) of all characteristic points. For example, any bounded domain of class $C^{1}$ in the Heisenberg group $\mathbb{H}^{n}$, whose boundary is homeomorphic to the $2n-$dimensional sphere $\mathbb{S}^{2n}$, has non-empty characteristic set (see, for example, \cite{DGN}).

We record relevant single and double layer potentials for the sub-Laplacian.
In \cite{J}, Jerison used the single layer potential defined by
$$
\mathcal{S}_0 u(x)=\int_{\partial\Omega} u(y) \varepsilon(y,x) dS(y),
$$
which, however, is not integrable over characteristic points. We refer to \cite{R91} for examples. On the contrary, the functionals
\begin{equation}\label{S}
\mathcal{S}_{j} u(x)=\int_{\partial\Omega} u(y) \varepsilon(y,x) \langle X_j, d\nu(y)\rangle,\quad j=1,...,N_{1},
\end{equation}
where $\langle X_{j}, d\nu\rangle$ is the canonical pairing between vector fields
and differential forms, are integrable over the whole boundary $\partial\Omega$ (see Lemma \ref{kl4}).
Parallel to $\mathcal{S}_{j}$, it will be natural to
use the operator
\begin{equation}\label{EQ:dp}
\mathcal{D}u(x)=\int_{\partial \Omega} u(y)\langle \widetilde{\nabla} \varepsilon(y,x),d\nu(y)\rangle,
\end{equation}
as a double layer potential, where $$\widetilde{\nabla} \varepsilon=\sum_{k=1}^{N_{1}}(X_{k}\varepsilon) X_{k},$$ with $X_{k}$ acting on the $y$-variable.

So let us define a family of single layer potentials
by Formula \eqref{S}. First we will prove the following lemma.
\begin{lem}\label{kl4}
Let $\partial\Omega$ be the boundary of an admissible domain $\Omega\subset\mathbb{G}$.
Then
$$\int_{\partial\Omega}[d(x,y)]^{2-Q}\langle X_{j}, d\nu(y)\rangle$$
is a convergent integral for any $x\in\mathbb{G}$ and $x\not\in\partial\Omega$.
\end{lem}
\begin{proof} [Proof of Lemma \ref{kl4}]
We have
\begin{multline*}
\int_{\partial\Omega}[d(x,y)]^{2-Q}\langle X_{j}, d\nu(y)\rangle
\\
=\int_{\Omega}X_{j}[d(x,y)]^{2-Q}d\nu(y)\leq \int_{\Omega}\mid X_{j}[d(x,y)]^{2-Q} \mid d\nu(y)
\\
\leq
\int_{B_{R}}\mid X_{j}[d(x,y)]^{2-Q} \mid d\nu(y)=C\int_{0}^{R}r^{1-Q}r^{Q-1}dr<\infty.
\end{multline*}
where
$B_{R}:=\{y: d(x,y)<R\}$ is a ball such that $\Omega\subset B_{R}.$
\end{proof}

\begin{thm}\label{singlelayer}
Let $\partial\Omega$ be the boundary of an admissible domain $\Omega\subset\mathbb{G}$. Let $u$ be bounded on $\partial\Omega$, that is, $u\in L^{\infty}(\partial \Omega)$. Then the single layer potential $\mathcal{S}_{j}u$ is continuous on $\mathbb G$, for all $j=1,\ldots,N_{1}$.
\end{thm}

\begin{proof}[Proof of Theorem \ref{singlelayer}.]
Let $x\in \mathbb{G}, x_{0}\not\in\partial\Omega,$ then
\begin{multline*}
| \mathcal{S}_{j}u(x) - \mathcal{S}_{j}u(x_{0})   | =  | \int _{\partial \Omega } u(y)(\varepsilon(y,x)  - \varepsilon(y,x_{0})) \langle X_{j},d\nu(y)  \rangle |
\\
\leq \underset{y\in\partial\Omega}{\sup}  | u(y)| | \int _{\partial \Omega} \varepsilon(y,x)  - \varepsilon(y,x_{0})\langle X_{j},d\nu(y) \rangle|.
\end{multline*}
This means
$$\underset{x\rightarrow x_{0}}{\lim}\mathcal{S}_{j}u(x)=\mathcal{S}_{j}u(x_{0}),$$
that is, the single layer potential $\mathcal{S}_{j}u$ is continuous in
$\mathbb G\backslash\partial\Omega.$
Now let $x\in \mathbb{G}, x_{0}\in\partial\Omega.$
Let
$\Omega_{\epsilon}:=\{y\in \Omega: d(x_{0},y)<\epsilon\}$.
Then
$$
| \mathcal{S}_{j}u(x) - \mathcal{S}_{j}u(x_{0})   | =  | \int _{\partial \Omega } u(y)(\varepsilon(y,x)  - \varepsilon(y,x_{0})) \langle X_{j},d\nu(y)  \rangle |
$$
$$
=  \underset{y\in\partial\Omega}{\sup}  | u(y)|\, | \int _{\partial \Omega }
(\varepsilon(y,x)  - \varepsilon(y,x_{0})) \langle X_{j},d\nu(y)  \rangle |
$$
$$
=  \underset{y\in\partial\Omega}{\sup}  | u(y)| \,| \int _{\Omega } X_{j}(\varepsilon(y,x)  - \varepsilon(y,x_{0})) d\nu(y)|
$$
$$
\leq \underset{y\in\partial\Omega}{\sup}  | u(y)| \;\underset{\epsilon\rightarrow 0}{\lim}\bigg(|\int _{\Omega\backslash\Omega_{\epsilon}} X_{j}(\varepsilon(y,x)  - \varepsilon(y,x_{0})) d\nu(y)|$$
$$+|\int _{\Omega_{\epsilon}}X_{j}(\varepsilon(y,x)  - \varepsilon(y,x_{0})) d\nu(y)|\bigg),
$$
where the first term tends to zero when $x\rightarrow x_{0}$.
Now it is left to show that the second term tends to zero, and we have
\begin{multline*}
\underset{\epsilon\rightarrow 0}{\lim}\int _{\Omega_{\epsilon}}X_{j}(\varepsilon(y,x)  - \varepsilon(y,x_{0})) d\nu(y)
\\
=\underset{\epsilon\rightarrow 0}{\lim}\left(\int _{\Omega_{\epsilon}} X_{j}\varepsilon(y,x)d\nu(y) - \int _{\Omega_{\epsilon}}X_{j}\varepsilon(y,x_{0}) d\nu(y)\right)
\\
=\underset{\epsilon\rightarrow 0}{\lim}(C\int _{0}^{\epsilon}r^{1-Q}r^{Q-1}dr)=C\,\underset{\epsilon\rightarrow 0}{\lim}\epsilon=0.
\end{multline*}
This completes the proof.
\end{proof}

We will prove the following lemma to prepare for establishing analogues of the Plemelj jump relations for the double layer potential $\mathcal{D}$ defined in \eqref{EQ:dp}.

\begin{lem}\label{claim2}
Let $u \in C^{1}(\Omega)\bigcap C(\overline{\Omega})$, $\Omega\subset\mathbb{G}$, be an admissible domain with the boundary $\partial\Omega$ and $x_{0}\in\partial\Omega$. Then
\begin{multline*}
\underset{x\rightarrow x_{0}}{\lim}\int_{\partial \Omega} [u(y) - u(x)]\langle \widetilde{\nabla}  \varepsilon(y,x),d\nu(y)\rangle \\
=\int_{\partial \Omega}[u(y)-u(x_{0})] \langle \widetilde{\nabla} \varepsilon(y,x_{0}),d\nu(y)\rangle,\quad x_{0}\in\partial\Omega,
\end{multline*}
where $\widetilde{\nabla}  \varepsilon=\sum_{k=1}^{N_{1}}\left(X_{k}\varepsilon\right)X_{k}.$
\end{lem}
\begin{proof}[Proof of Lemma \ref{claim2}.]
In this proof we use the Einstein type notation, that is, if the index $k$ is repeated in an integrant,
then it means that we have a sum from $1$ to $N_{1}$ over $k$ (and both indices can enter as subscripts). For example,
$$\int_{ \partial \Omega} [u(y) - u(x)]  X_{k}\varepsilon(y,x) \langle X_{k},d\nu(y)  \rangle :=\sum_{k=1}^{N_{1}}\int_{ \partial \Omega} [u(y) - u(x)]  X_{k}\varepsilon(y,x) \langle X_{k},d\nu(y)  \rangle.$$
First, let us show that
$$
\underset{\epsilon\rightarrow 0}{\lim} \; \int_{d(x,y)<\epsilon}  X_{k}\left\{(u(y) - u(x)) X_{k}\varepsilon(y,x)\right\} d\nu(y)=0.
$$
By using Proposition \ref{stokes}
$$
\underset{\epsilon\rightarrow 0}{\lim} \;| \int_{d(x,y)<\epsilon}  X_{k}\left\{(u(y) - u(x)) X_{k}\varepsilon(y,x)\right\} d\nu(y)|
$$
$$
\leq C_{1} \,\underset{\epsilon\rightarrow 0}{\lim} \int_{d(x,y)<\epsilon} |X_{k}\varepsilon(y,x)| d\nu(y)
$$
$$
+\,\underset{\epsilon\rightarrow 0}{\lim} \int_{d(x,y)<\epsilon}  |(u(y) - u(x)) X_{k}X_{k}\varepsilon(y,x)| d\nu(y)
$$
$$\leq C\,\underset{\epsilon\rightarrow 0}{\lim} \int_{0}^{\epsilon}dr=0.$$
Therefore, we have
$$\int_{  \Omega} X_{k}\left\{ [u(y) - u(x)]  X_{k}\varepsilon(y,x)\right\}d\nu(y) $$
$$=\int_{  \Omega} X_{k}\left\{ [u(y) - u(x)]  X_{k}\varepsilon(y,x)\right\}d\nu(y)$$
$$+\, \underset{\epsilon\rightarrow 0}{\lim} \; \int_{d(x,y)<\epsilon}  X_{k}\left\{(u(y) - u(x)) X_{k}\varepsilon(y,x)\right\} d\nu(y).$$
If we take $\Omega_{\epsilon} = \{ y\in\mathbb G:d(x,y) < \epsilon  \}$, then
$$ \underset{x\rightarrow x_{0}}{\lim}\int_{ \partial \Omega} [u(y) - u(x)]  X_{k}\varepsilon(y,x) \langle X_{k},d\nu(y)  \rangle   $$
and by the Divergence formula (see Proposition \ref{stokes}) the above expression is
$$ =  \underset{x\rightarrow x_{0}}{\lim}\int_{  \Omega} X_{k}\left\{ [u(y) - u(x)]  X_{k}\varepsilon(y,x)\right\}d\nu(y) $$

$$= \underset{x\rightarrow x_{0}}{\lim} \underset{\epsilon\rightarrow 0}{\lim}
 \bigg\{ \int_{ \Omega\backslash \Omega_{\epsilon}} X_{k} \left\{ [u(y) - u(x)]  X_{k}\varepsilon(y,x) \right\} d\nu(y)
$$
$$+ \int_{  \Omega_{\epsilon}} X_{k}\left\{ [u(y) - u(x)]  X_{k}\varepsilon(y,x)\right\}d\nu(y) \bigg\}$$

$$ = \int_{  \Omega} X_{k} \left\{ [u(y) - u(x_{0})]  X_{k}\varepsilon(y,x_{0})\right\}d\nu(y).$$
That is,
$$\label{e1} \underset{x\rightarrow x_{0}}{\lim}\int_{ \partial \Omega} [u(y) - u(x)]  X_{k}\varepsilon(y,x) \langle X_{k},d\nu(y)\rangle$$
$$=\int_{\Omega} X_{k} \left\{ [u(y) - u(x_{0})]  X_{k}\varepsilon(y,x_{0})\right\}d\nu(y)$$
$$=\int_{\partial \Omega}[u(y)-u(x_{0})] X_{k}\varepsilon(y,x_{0})\langle X_{k},d\nu(y)\rangle.$$
As we agreed this is the same as

$$\underset{x\rightarrow x_{0}}{\lim}\int_{\partial \Omega} [u(y) - u(x)]\langle \widetilde{\nabla}  \varepsilon(y,x),d\nu(y)\rangle$$
$$
=\int_{\partial \Omega}[u(y)-u(x_{0})] \langle \widetilde{\nabla} \varepsilon(y,x_{0}),d\nu(y)\rangle,\quad x_{0}\in\partial\Omega,
$$
where $\widetilde{\nabla}  \varepsilon=\sum_{k=1}^{N_{1}}\left(X_{k}\varepsilon\right)X_{k}.$
\end{proof}

We now prove the Plemelj type jump relations for the double layer potential $\mathcal{D}$ in \eqref{EQ:dp}.

\begin{thm}\label{doublelayer}
Let $\Omega\subset\mathbb{G}$ be an admissible domain and let $u \in C^{1}(\Omega)\bigcap C(\overline{\Omega})$. Define

$$
\mathcal{D}^{0}u(x_{0}):=
\int_{\partial \Omega} u(y)\langle \widetilde{\nabla}  \varepsilon(y,x_{0}),d\nu(y)\rangle,
$$
$$
\mathcal{D}^{+}u(x_{0}):= \underset{x\rightarrow x_{0},\,x \in \Omega}{\lim} \int_{\partial \Omega} u(y )\langle \widetilde{\nabla}  \varepsilon(y,x),d\nu(y) \rangle
$$
and
$$
\mathcal{D}^{-}u(x_{0}):= \underset{x\rightarrow x_{0},\, x \notin \overline{\Omega}}{\lim} \int_{\partial \Omega} u(y )\langle \widetilde{\nabla}  \varepsilon(y,x),d\nu(y)\rangle,
$$
for $x_{0} \in\partial \Omega$.
Then $\mathcal{D}^{+}u(x_{0}), \mathcal{D}^{-}u(x_{0})$ and $\mathcal{D}^{0}u(x_{0})$ exist and verify the following jump relations:
$$
\mathcal{D}^{+}u(x_{0}) - \mathcal{D}^{-}u(x_{0}) = u(x_{0}),
$$
$$
\mathcal{D}^{0}u(x_{0}) - \mathcal{D}^{-}u(x_{0}) = \mathcal{J}(x_{0})u(x_{0}),
$$
$$
\mathcal{D}^{+}u(x_{0}) - \mathcal{D}^{0}u(x_{0}) = (1 - \mathcal{J}(x_{0}))u(x_{0}),
$$
where the jump value $\mathcal{J}(x_{0})$ is given by the formula
$$
\mathcal{J}(x_{0})  =  \int_{\partial \Omega} \langle \widetilde{\nabla}  \varepsilon(y,x_{0}),d\nu(y)\rangle,\quad x_{0} \in\partial \Omega,
$$
in the sense of the (Cauchy) principal value
and $\widetilde{\nabla}  \varepsilon=\sum_{k=1}^{N_{1}}\left(X_{k}\varepsilon\right)X_{k}.$
\end{thm}
\begin{proof}[Proof of Theorem \ref{doublelayer}.]
We have
\begin{multline*}
\underset{x\rightarrow x_{0},\,x \notin \partial\Omega}{\lim} \int_{\partial \Omega} u(y )\langle \widetilde{\nabla}  \varepsilon(y,x),d\nu(y)  \rangle
= \\
\underset{x\rightarrow x_{0},\,x \notin \partial\Omega}{\lim} (\int_{\partial \Omega} [u(y) - u(x) ]\langle \widetilde{\nabla}  \varepsilon(y,x),d\nu(y)\rangle
+ u(x )\int_{\partial \Omega} \langle \widetilde{\nabla}  \varepsilon(y,x),d\nu(y) \rangle ).
\end{multline*}
Taking $u=\varepsilon$ and $v=1$ in the Green's first formula \eqref{g1} we get
$$
\int_{\partial \Omega} \langle \widetilde{\nabla}  \varepsilon(y,x),d\nu(y)  \rangle  =
\bigg\{\begin{matrix}
1, &x \in \Omega, \\ 0,
  & x \notin \bar\Omega.
\end{matrix}
$$
Therefore, using Lemma \ref{claim2} we obtain
\begin{equation}\label{{D}^{+}}
\mathcal{D}^{+}u(x_{0}) = \int_{\partial \Omega} [u(y) - u(x_{0})]\langle \widetilde{\nabla}  \varepsilon(y,x_{0}),d\nu(y)\rangle +u(x_{0}),
\end{equation}

\begin{equation}\label{{D}^{-}}
\mathcal{D}^{-}u(x_{0}) = \int_{\partial \Omega} [u(y) - u(x_{0})]\langle \widetilde{\nabla}  \varepsilon(y,x_{0}),d\nu(y)\rangle.
\end{equation}
From here we obtain the first jump relation, i.e.
$$\mathcal{D}^{+}u(x_{0})-\mathcal{D}^{-}u(x_{0})=u(x_{0}).$$
We also have
$$
\mathcal{D}^{0}u(x_{0}) = \int_{\partial \Omega} u(y)\langle \widetilde{\nabla}  \varepsilon(y,x_{0}),d\nu(y) \rangle
$$
$$
= \int_{\partial \Omega } [u(y) - u(x_{0})] \langle \widetilde{\nabla}  \varepsilon(y,x_{0}),d\nu(y)  \rangle
$$
$$
+u(x_{0})\,\int_{\partial \Omega }  \langle \widetilde{\nabla}  \varepsilon(y,x_{0}),d\nu(y) \rangle
$$
$$
= \int_{\partial \Omega}[u(y) - u(x_{0})]\left \langle \widetilde{\nabla}  \varepsilon(y,x_{0}),d\nu(y) \right \rangle + \mathcal{J}(x_{0})u(x_{0}).
$$
So we obtain
\begin{equation}\label{{D}^{0}}
\mathcal{D}^{0}u(x_{0})=\int_{\partial \Omega}[u(y) - u(x_{0})]
\left \langle \widetilde{\nabla}  \varepsilon(y,x_{0}),d\nu(y) \right \rangle + \mathcal{J}(x_{0})u(x_{0}).
\end{equation}
Now subtracting \eqref{{D}^{-}} from \eqref{{D}^{0}} we get the second jump relation and
subtracting \eqref{{D}^{0}} from \eqref{{D}^{+}} we obtain the third one.
\end{proof}

\section{Traces and Kac's problem for the sub-Laplacian}
\label{SEC:5}

For $0<\alpha<1$, Folland and Stein
(see \cite{FS} and see also \cite{Fol75}) defined the anisotropic H\"older spaces $\Gamma_\alpha(\Omega),$ $\Omega\subset\mathbb{G},$
by
$$
\Gamma_\alpha(\Omega)=\{
f:\Omega\to\mathbb C:\; \sup_{\stackrel{x,y\in \Omega}{x\not= y}}
\frac{|f(x)-f(y)|}{[d(x,y)]^\alpha}<\infty
\}.
$$
For $k\in\mathbb N$ and $0<\alpha<1$, one defines
$\Gamma_{k+\alpha}(\Omega)$ as the space of all $f:\Omega\to\mathbb C$ such that all derivatives of $f$ of order $k$ belong to $\Gamma_\alpha(\Omega)$.
A bounded function $f$ is called  $\alpha$-H\"{o}lder continuous
in $\Omega\subset\mathbb{G}$ if $f\in\Gamma_{\alpha}(\Omega)$.

Let $\Omega\subset \mathbb{G}$ be an admissible domain.
Consider the following analogy of the Newton potential
\begin{equation}
u(x)=\int_{\Omega}f(y)
\varepsilon(y,x)d\nu(y),\quad x\in\Omega,\quad f\in \Gamma_\alpha(\Omega),
\label{6}
\end{equation}
where $$\varepsilon(y,x)=\varepsilon(x,y)=\varepsilon(y^{-1}x,0)=\varepsilon(y^{-1}x)$$ is the
fundamental solution \eqref{fundsol} of the sub-Laplacian $\mathcal{L}$, i.e.
$$\varepsilon(x,y)=[d(x,y)]^{2-Q},$$ and
$u$ is a solution of $$\mathcal{L} u= f$$ in $\Omega$. The aim of this section is to find a boundary condition
 for $u$ such that with this boundary condition the equation $\mathcal{L} u= f$ has a unique solution in $C^{2}(\Omega)$, say, and this solution is the Newton potential \eqref{6}.
 This amounts to finding the trace of the integral operator in \eqref{6} on $\partial\Omega$.

\medskip

A starting point for us will be that if $f\in \Gamma_\alpha(\Omega)$ for  $\alpha>0$ then
$u$ defined by \eqref{6} is twice differentiable and satisfies the
equation $\mathcal{L} u= f$. We refer to Folland \cite{Fol75} for this property.
These results extend those known for the Laplacian, in suitably redefined anisotropic H\"older spaces.

Our main result for the sub-Laplacian is the following variant of formula \eqref{16} in the introduction, now in the setting of Carnot groups.

\begin{thm} \label{THM:main}
Let $\varepsilon(y,x)=\varepsilon(y^{-1}x)$ be the  fundamental solution to
$\mathcal{L}$, so that
\begin{equation}\label{EQ:def-eps}
\mathcal{L}\varepsilon=\delta\quad \textrm{ on } \mathbb{G}.
\end{equation}
Let $\Omega\subset\mathbb{G}$ be an admissible domain. For any $f\in \Gamma_\alpha(\Omega)$, $0<\alpha<1$, $\textrm{supp}  f \subset \Omega$,
the Newton potential \eqref{6} is the unique solution
in $C^{2}(\Omega)\cap C^1(\overline{\Omega})$
of the equation
\begin{equation}
\mathcal{L} u= f\,\,\,\textrm{ in } \Omega,
\label{EQ:BV1}
\end{equation}
with the boundary condition
\begin{multline}\label{7}
(1-\mathcal{J}(x))u(x)+\int_{\partial \Omega} u(y)\langle \widetilde{\nabla} \varepsilon(y,x),d\nu(y)\rangle
\\
-\int_{\partial \Omega} \varepsilon(y,x) \langle \widetilde{\nabla}  u(y), d\nu(y)\rangle =0,
\qquad \textrm{for}\quad x\in\partial\Omega,
\end{multline}
where the jump value is given by the formula
\begin{equation}\label{EQ:HR}
\mathcal{J}(x)=\int_{\partial \Omega}
\langle \widetilde{\nabla} \varepsilon(y,x),d\nu(y)\rangle,
\end{equation}
with $\widetilde{\nabla}=\widetilde{\nabla}_{y}$ defined by
$$\widetilde{\nabla}  g=\sum_{k=1}^{N_{1}} (X_{k}g) X_{k}.$$
\end{thm}

\begin{proof}[Proof of Theorem \ref{THM:main}]
Since the Newton potential
\begin{equation}\label{EQ:solidp}
u(x)=\int_{\Omega}f(y)\varepsilon(y,x)d\nu(y)
\end{equation}
is a solution of \eqref{EQ:BV1}, from the aforementioned results of Folland it follows that
$u$ is locally in $\Gamma_{\alpha+2}(\Omega,loc)$ and that it is twice differentiable in
$\Omega$. In particular, it follows that $u\in C^{2}(\Omega)\cap C^1(\overline{\Omega})$.

By Corollary \ref{repn} we have the following representation formula
\begin{multline}\label{8}
u(x)=\int_{\Omega}f(y)\varepsilon(y,x)d\nu(y)
+\int_{\partial \Omega} u(y)\langle \widetilde{\nabla} \varepsilon(y,x),d\nu(y)\rangle
\\
-\int_{\partial \Omega} \varepsilon(y,x) \langle\widetilde{\nabla}  u(y),d\nu(y)\rangle,
\quad \forall x\in\Omega,
\end{multline}
for $u\in C^{2}(\Omega)\cap C^1(\overline{\Omega})$.

Since $u(x)$ given by \eqref{EQ:solidp}
is a solution of \eqref{EQ:BV1}, using it in  \eqref{8} we get
\begin{equation}\label{EQ:aux1}
\int_{\partial \Omega} u(y)\langle \widetilde{\nabla} \varepsilon(y,x),d\nu(y)\rangle-\int_{\partial \Omega} \varepsilon(y,x) \langle\widetilde{\nabla}  u(y),d\nu(y)\rangle=0, \quad
\textrm{for any}\,\,x\in\Omega.
\end{equation}
By using Theorem \ref{singlelayer} and Theorem \ref{doublelayer} as $x\in\Omega$ approaches the boundary
$\partial\Omega$ from inside, we find that
\begin{multline}\label{EQ:BC}
(1-\mathcal{J}(x))u(x)+\int_{\partial \Omega} u(y)\langle \widetilde{\nabla} \varepsilon(y,x),d\nu(y)\rangle
\\
-\int_{\partial \Omega} \varepsilon(y,x) \langle\widetilde{\nabla}  u(y),d\nu(y)\rangle=0,
\quad \textrm{ for any } x\in\partial\Omega.
\end{multline}

This shows that \eqref{6} is a solution of the boundary value problem \eqref{EQ:BV1} with the boundary condition
\eqref{7}.

\medskip
Now let us prove its uniqueness.
If the boundary value problem has two solutions $u$ and $u_{1}$
then the function
$$w=u-u_{1}\in C^{2}(\Omega)\cap C^1(\overline{\Omega})$$
 satisfies the homogeneous equation
\begin{equation}
\mathcal{L}w=0\,\,\,\textrm{ in } \Omega, \label{9}
\end{equation}
and the boundary condition \eqref{7}, i.e.
\begin{multline}
(1-\mathcal{J}(x))w(x)+\int_{\partial \Omega} w(y)\langle \widetilde{\nabla} \varepsilon(y,x),d\nu(y)\rangle\label{10} \\
-\int_{\partial \Omega} \varepsilon(y,x) \langle\widetilde{\nabla}  w(y),d\nu(y)\rangle=0,\quad x\in\partial\Omega.
\end{multline}
Since $f\equiv0$ in this case instead of \eqref{8} we have the following representation formula (see Corollary \ref{repnhar})
\begin{equation}
w(x)=\int_{\partial \Omega} w(y)\langle \widetilde{\nabla} \varepsilon(y,x),d\nu(y)\rangle-
\int_{\partial \Omega} \varepsilon(y,x) \langle\widetilde{\nabla}  w(y),d\nu(y)\rangle, \label{11}\end{equation}
for any $x\in\Omega$.
As above, by using the properties of the double and single layer potentials as $x\rightarrow \partial\Omega$ from interior, from \eqref{11} we obtain
\begin{multline}
w(x)=(1-\mathcal{J}(x))w(x)
\label{12}\\
+\int_{\partial \Omega } w(y)\langle \widetilde{\nabla} \varepsilon(y,x),d\nu(y)\rangle-
\int_{\partial \Omega} \varepsilon(y,x) \langle\widetilde{\nabla}  w,d\nu(y)\rangle,
\end{multline}
for any $x\in\partial\Omega.$
Comparing this with \eqref{10} we arrive at
\begin{equation}
w(x)=0,  \,\, x\in\partial\Omega.\label{13}
\end{equation}

The homogeneous equation \eqref{9} with the Dirichlet boundary condition \eqref{13} has only trivial solution
$w\equiv 0$ in $\Omega$, see Lemma \ref{uniqueness}.
This shows that the boundary value problem \eqref{EQ:BV1} with the boundary condition
\eqref{7} has a unique solution in $C^{2}(\Omega)\cap C^1(\overline{\Omega})$.
This completes the proof of Theorem \ref{THM:main}.
\end{proof}

\begin{rem}
It follows from Theorem \ref{THM:main} that the kernel $\varepsilon(y,x)=\varepsilon(y^{-1}x)$, which is a
fundamental solution of the sub-Laplacian, is the Green function of the boundary value problem \eqref{EQ:BV1}, \eqref{7} in $\Omega$. Therefore, the
boundary value problem \eqref{EQ:BV1}, \eqref{7} can serve as an example of an explicitly solvable boundary value problem for the sub-Laplacian in any (admissible)
domain $\Omega$ on the homogeneous Carnot  group.
\end{rem}

\section{Powers of the sub-Laplacian}
\label{SEC:6}

As before, let $\Omega \subset \mathbb{G}$ be
an admissible domain.
For $m\in\mathbb N$, we denote $\mathcal{L}^{m}:=\mathcal{L}\mathcal{L}^{m-1}$.
Then for $m=1,2,\ldots$, we consider
the equation
\begin{equation}
\mathcal{L}^{m}u(x)=f(x), \,\,x\in\Omega,
\label{17}
\end{equation}
for a given $f\in \Gamma_{\alpha}(\Omega).$
Let $\varepsilon(y,x)=\varepsilon(y^{-1}x)$ be the
 fundamental solution of the sub-Laplacian as in \eqref{EQ:def-eps}.
Let us now define
\begin{equation}
u(x)=\int_{\Omega}f(y)\varepsilon_{m}(y,x)d\nu(y)
\label{18}
\end{equation}
in $\Omega\subset\mathbb{G}$, where
$\varepsilon_{m}(y,x)$ is a  fundamental solution of \eqref{17} such that
$$
\mathcal{L}^{m-1}\varepsilon_{m}=\varepsilon, \quad m=1,2,....
$$
We take, with a proper distributional interpretation, for $m=2,3,\ldots$,
\begin{equation}\varepsilon_{m}(y,x)=\int_{\Omega}\varepsilon_{m-1}(y,\zeta)\varepsilon(\zeta,x)d\nu(\zeta),\qquad
y,x\in \Omega, \label{19}\end{equation}
with
$$\varepsilon_{1}(y,x)=\varepsilon(y,x).$$

A simple calculation shows that the generalised Newton potential \eqref{18} is a solution of \eqref{17}
in $\Omega$. The aim of this section is to find boundary conditions on $\partial\Omega$ such that with these boundary
conditions the equation \eqref{17} has a unique solution in $C^{2m}(\Omega)$, which coincides with \eqref{18}.

Although higher order hypoelliptic operators on the homogenous Carnot
group may not have unique fundamental solutions, see Geller \cite{geller_83}, in the case of
the iterated sub-Laplacian $\mathcal{L}^{m}$ we still have the uniqueness for our
problem in the sense of the following theorem, and the uniqueness argument in its proof.

\begin{thm} \label{THM:main2}
Let $\Omega\subset\mathbb{G}$ be an admissible domain. For any $f\in \Gamma_{\alpha}(\Omega)$, $0<\alpha<1$, $\textrm{supp}  f \subset \Omega$, the generalised Newton potential \eqref{18} is a unique solution of the equation \eqref{17} in $C^{2m}(\Omega)\cap C^{2m-1}(\overline{\Omega})$
with $m$ boundary conditions
\begin{multline}
(1-\mathcal{J}(x))\mathcal{L}^{i}u(x) +
\sum_{j=0}^{m-i-1}\int_{\partial \Omega }
\mathcal{L}^{j+i}u(y)\langle\widetilde{\nabla} \mathcal{L}^{m-1-j}\varepsilon_{m}(y,x),d\nu(y)\rangle
\\
-\sum_{j=0}^{m-i-1}\int_{\partial\Omega}\mathcal{L}^{m-1-j}
\varepsilon_{m}(y,x)\langle\widetilde{\nabla} \mathcal{L}^{j+i}u(y)d\nu(y)\rangle=0,\quad
x\in \partial\Omega,
\label{20}
\end{multline}
for all $i=0,1,\ldots,m-1,$
where $\widetilde{\nabla}$ is given by
$$\widetilde{\nabla}  g=\sum_{k=1}^{N_{1}} (X_{k}g)X_{k}$$
and $\mathcal{J}(x)$ is the jump function given by the formula \eqref{EQ:HR}.
\end{thm}

\begin{proof}[Proof of Theorem \ref{THM:main2}.]
By applying Green's second formula for each $x\in \Omega$,
as in \eqref{8} (see Proposition \ref{green2}), we obtain
$$u(x)=\int_{\Omega}f(y)\varepsilon_{m}(y,x)d\nu(y)$$
$$=\int_{\Omega}\mathcal{L}^{m}u(y)\varepsilon_{m}(y,x)d\nu(y)$$

$$=\int_{\Omega}\mathcal{L}^{m-1}u(y)\mathcal{L}\varepsilon_{m}(y,x)d\nu(y)$$
$$-
\int_{\partial\Omega}\mathcal{L}^{m-1}u(y)\langle \widetilde{\nabla} \varepsilon_{m}(y,x),d\nu(y)\rangle$$

$$+\int_{\partial\Omega}\varepsilon_{m}(y,x)\langle \widetilde{\nabla} \mathcal{L}^{m-1}u(y),d\nu(y)\rangle$$
$$=
\int_{\Omega}\mathcal{L}^{m-2}u(y)\mathcal{L}^{2}\varepsilon_{m}(y,x)d\nu(y)$$

$$-\int_{\partial\Omega}\mathcal{L}^{m-2}u(y)\langle\widetilde{\nabla} \mathcal{L}\varepsilon_{m}(y,x),d\nu(y)\rangle$$

$$+\int_{\partial\Omega}\mathcal{L}\varepsilon_{m}(y,x)\langle\widetilde{\nabla} \mathcal{L}^{m-2}u(y),d\nu(y)\rangle$$

$$-\int_{\partial\Omega}\mathcal{L}^{m-1}u(y)\langle \widetilde{\nabla} \varepsilon_{m}(y,x),d\nu(y)\rangle$$

$$+\int_{\partial\Omega}\varepsilon_{m}(y,x)\langle \widetilde{\nabla} \mathcal{L}^{m-1}u(y),d\nu(y)\rangle=...$$

$$=u(x)-\sum_{j=0}^{m-1}\int_{\partial\Omega}\mathcal{L}^{j}u(y)\langle\widetilde{\nabla} \mathcal{L}^{m-1-j}\varepsilon_{m}(y,x),d\nu(y)\rangle$$
$$
+\sum_{j=0}^{m-1}\int_{\partial\Omega}\mathcal{L}^{m-1-j}\varepsilon_{m}(y,x)\langle\widetilde{\nabla} \mathcal{L}^{j}u(y),d\nu(y)\rangle,\quad x\in \Omega.
$$
This implies the identity
\begin{multline}\label{22}
\sum_{j=0}^{m-1}\int_{\partial\Omega}\mathcal{L}^{j}u(y)\langle\widetilde{\nabla}
\mathcal{L}^{m-1-j}\varepsilon_{m}(y,x),d\nu(y)\rangle\\
-\sum_{j=0}^{m-1}\int_{\partial\Omega}\mathcal{L}^{m-1-j}\varepsilon_{m}(y,x)\langle\widetilde{\nabla} \mathcal{L}^{j}u(y),d\nu(y)\rangle=0,\quad x\in \Omega.
\end{multline}
Note that here only the first term of the first summand, i.e.,
$j=0$ term of the first summand:

$$\int_{\partial\Omega}u(y)\langle\widetilde{\nabla}
\varepsilon(y,x),d\nu(y)\rangle$$
is the double layer potential (see Theorem \ref{doublelayer}). The other terms of the summands are single layer type potentials, that is, they are continuous functions on $\mathbb{G}$ (see Theorem \ref{singlelayer}).
By using the properties of the double and single layer potentials as $x$
approaches the boundary $\partial\Omega$ from the interior, from \eqref{22} we obtain
\begin{multline*}
(1-\mathcal{J}(x))u(x)+\sum_{j=0}^{m-1}\int_{\partial \Omega }\mathcal{L}^{j}u(y)\langle\widetilde{\nabla}
\mathcal{L}^{m-1-j}\varepsilon_{m}(y,x),d\nu(y)\rangle \\
-\sum_{j=0}^{m-1}\int_{\partial\Omega}\mathcal{L}^{m-1-j}\varepsilon_{m}(y,x)\langle\widetilde{\nabla} \mathcal{L}^{j}u(y),d\nu(y)\rangle=0,\quad x\in \partial\Omega.
\end{multline*}
Thus, this relation is one of the boundary conditions of \eqref{18}.
Let us derive the remaining boundary conditions. To this end, we write
\begin{equation}
\mathcal{L}^{m-i}\mathcal{L}^{i}u=f,\quad i=0,1,\ldots,m-1,\quad m=1,2,\ldots,
\label{24}
\end{equation}
and carry out similar considerations just as above. This yields

$$\mathcal{L}^{i}u(x)=\int_{\Omega}f(y)\mathcal{L}^{i}\varepsilon_{m}(y,x)d\nu(y)$$
$$=\int_{\Omega}\mathcal{L}^{m-i}\mathcal{L}^{i}u(y)\mathcal{L}^{i}\varepsilon_{m}(y,x)d\nu(y)$$
$$=\int_{\Omega}\mathcal{L}^{m-i-1}\mathcal{L}^{i}u(y)\mathcal{L}\mathcal{L}^{i}\varepsilon_{m}(y,x)d\nu(y)$$
$$-\int_{\partial\Omega}\mathcal{L}^{m-i-1}\mathcal{L}^{i}u(y)\langle\widetilde{\nabla} \mathcal{L}^{i}\varepsilon_{m}(y,x),d\nu(y)\rangle$$
$$+\int_{\partial\Omega}\mathcal{L}^{i}\varepsilon_{m}(y,x)\langle \widetilde{\nabla} \mathcal{L}^{m-i-1}\mathcal{L}^{i}u(y),d\nu(y)\rangle$$
$$=\int_{\Omega}\mathcal{L}^{m-i-2}\mathcal{L}^{i}u(y)\mathcal{L}^{2}\mathcal{L}^{i}\varepsilon_{m}(y,x)d\nu(y)$$
$$-\int_{\partial\Omega}\mathcal{L}^{m-i-2}\mathcal{L}^{i}u(y)\langle\widetilde{\nabla} \mathcal{L}\mathcal{L}^{i}\varepsilon_{m}(y,x),d\nu(y)\rangle$$
$$+\int_{\partial\Omega}\mathcal{L}\mathcal{L}^{i}\varepsilon_{m}(y,x)\langle\widetilde{\nabla} \mathcal{L}^{m-i-2}\mathcal{L}^{i}u(y),d\nu(y)\rangle$$
$$-\int_{\partial\Omega}\mathcal{L}^{m-i-1}\mathcal{L}^{i}u(y)\langle\widetilde{\nabla} \mathcal{L}^{i}\varepsilon_{m}(y,x),d\nu(y)\rangle$$
$$+\int_{\partial\Omega}\mathcal{L}^{i}\varepsilon_{m}(y,x)\langle\widetilde{\nabla} \mathcal{L}^{m-i-1}\mathcal{L}^{i}u(y),d\nu(y)\rangle$$
$$=...=\int_{\Omega}\mathcal{L}^{i}u(y)\mathcal{L}^{m-i}\mathcal{L}^{i}\varepsilon_{m}(y,x)d\nu(y)$$
$$-\sum_{j=0}^{m-i-1}\int_{\partial\Omega}\mathcal{L}^{j}\mathcal{L}^{i}u(y)\langle\widetilde{\nabla} \mathcal{L}^{m-i-1-j}\mathcal{L}^{i}\varepsilon_{m}(y,x),d\nu(y)\rangle$$
$$+\sum_{j=0}^{m-i-1}\int_{\partial\Omega}\mathcal{L}^{m-i-1-j}\mathcal{L}^{i}\varepsilon_{m}(y,x)
\langle\widetilde{\nabla} \mathcal{L}^{j}\mathcal{L}^{i}u(y),d\nu(y)\rangle$$
$$=\mathcal{L}^{i}u(x)-\sum_{j=0}^{m-i-1}\int_{\partial\Omega}\mathcal{L}^{j+i}u(y)\langle\widetilde{\nabla} \mathcal{L}^{m-1-j}\varepsilon_{m}(y,x),d\nu(y)\rangle$$
$$+\sum_{j=0}^{m-i-1}\int_{\partial\Omega}\mathcal{L}^{m-1-j}\varepsilon_{m}(y,x)\langle\widetilde{\nabla} \mathcal{L}^{j+i}u(y),d\nu(y)\rangle,\quad x\in\Omega,$$
where $\mathcal{L}^{i}\varepsilon_{m}$ is a fundamental solution of the equation \eqref{24}, i.e.,
$$
\mathcal{L}^{m-i}\mathcal{L}^{i}\varepsilon_{m}= \delta,\qquad i=0,1,\ldots,m-1.
$$
From the previous relations, we obtain the identities
\begin{multline*}
\sum_{j=0}^{m-i-1}\int_{\partial\Omega}
\mathcal{L}^{j+i}u(y)\langle\widetilde{\nabla} \mathcal{L}^{m-1-j}\varepsilon_{m}(y,x),d\nu(y)\rangle
\\
-\sum_{j=0}^{m-i-1}\int_{\partial\Omega}\mathcal{L}^{m-1-j}\varepsilon_{m}(y,x)
\langle\widetilde{\nabla} \mathcal{L}^{j+i}u(y),d\nu(y)\rangle=0
\end{multline*}
for any $x\in\Omega$ and $i=0,1,\ldots,m-1.$
By using the properties of the double and single layer potentials as
$x$ approaches the boundary $\partial\Omega$ from the interior of $\Omega$, we find that
\begin{multline*}\label{25}
(1-\mathcal{J}(x))\mathcal{L}^{i}u(x)+\sum_{j=0}^{m-i-1}\int_{\partial \Omega }
\mathcal{L}^{j+i}u(y)\langle\widetilde{\nabla} \mathcal{L}^{m-1-j}\varepsilon_{m}(y,x),d\nu(y)\rangle
\\
-\sum_{j=0}^{m-i-1}\int_{\partial\Omega}\mathcal{L}^{m-1-j}\varepsilon_{m}(y,x)
\langle\widetilde{\nabla} \mathcal{L}^{j+i}u(y),d\nu(y)\rangle=0,\quad
x\in\partial\Omega,
\end{multline*}
are all boundary conditions of \eqref{18} for each $i=0,1,\ldots,m-1$.

\medskip
Conversely, let us show that if a function $w\in C^{2m}(\Omega)\cap C^{2m-1}(\overline{\Omega})$ satisfies the equation $\mathcal{L}^{m}w=f$  and
the boundary conditions \eqref{20}, then it coincides with the solution \eqref{18}.
Indeed, otherwise the function
$$v=u-w\in C^{2m}(\Omega)\cap C^{2m-1}(\overline{\Omega}),$$
where $u$ is the generalised Newton potential \eqref{18}, satisfies the homogeneous equation
\begin{equation}
\mathcal{L}^{m}v=0\label{26}
\end{equation}
and the boundary conditions \eqref{20}, i.e.
\begin{multline*}
I_{i}(v)(x):=
(1-\mathcal{J}(x))\mathcal{L}^{i}v(x)
+\sum_{j=0}^{m-i-1}\int_{\partial \Omega }
\mathcal{L}^{j+i}v(y)\langle\widetilde{\nabla} \mathcal{L}^{m-1-j}\varepsilon_{m}(y,x),d\nu(y)\rangle
\\ -
\sum_{j=0}^{m-i-1}\int_{\partial\Omega}\mathcal{L}^{m-1-j}\varepsilon_{m}(y,x)
\langle\widetilde{\nabla} \mathcal{L}^{j+i}v(y),d\nu(y)\rangle=0,\quad i=0,1,\ldots,m-1,
\end{multline*}
for $x\in\partial\Omega.$
By applying the Green formula to the function $v\in C^{2m}(\Omega)\cap C^{2m-1}(\overline{\Omega})$ and by following the lines of the above
argument, we obtain
$$
0=\int_{\Omega}\mathcal{L}^{m}v(x)\mathcal{L}^{i}\varepsilon_{m}(y,x)d\nu(y)$$
$$=\int_{\Omega}\mathcal{L}^{m-i}\mathcal{L}^{i}v(x)\mathcal{L}^{i}\varepsilon_{m}(y,x)d\nu(y)$$
$$=\int_{\Omega}\mathcal{L}^{m-1}v(x)\mathcal{L}\mathcal{L}^{i}\varepsilon_{m}(y,x)d\nu(y)$$
$$-\int_{\partial\Omega}\mathcal{L}^{m-1}v(x)\langle\widetilde{\nabla} \mathcal{L}^{i}\varepsilon_{m}(y,x),d\nu(y)\rangle$$
$$+\int_{\partial\Omega}\mathcal{L}^{i}\varepsilon_{m}(y,x)\langle\widetilde{\nabla} \mathcal{L}^{m-1}v(x),d\nu(y)\rangle$$

$$=\int_{\Omega}\mathcal{L}^{m-2}v(x)\mathcal{L}^{2}
\mathcal{L}^{i}\varepsilon_{m}(y,x)d\nu(y)$$

$$-\int_{\partial\Omega}\mathcal{L}^{m-2}v(x)\langle\widetilde{\nabla} \mathcal{L}^{i+1}\varepsilon_{m}(y,x),d\nu(y)\rangle$$
$$+\int_{\partial\Omega}\mathcal{L}^{i+1}\varepsilon_{m}(y,x)\langle\widetilde{\nabla} \mathcal{L}^{m-2}v(x),d\nu(y)\rangle$$

$$-\int_{\partial\Omega}\mathcal{L}^{m-1}v(x)\langle\widetilde{\nabla} \mathcal{L}^{i}\varepsilon_{m}(y,x),d\nu(y)\rangle$$
$$+\int_{\partial\Omega}\mathcal{L}^{i}\varepsilon_{m}(y,x)\langle\widetilde{\nabla} \mathcal{L}^{m-1}v(x),d\nu(y)\rangle=...$$

$$=\mathcal{L}^{i}v(x)-\sum_{j=0}^{m-i-1}\int_{\partial\Omega}\mathcal{L}^{j+i}v(y)\langle\widetilde{\nabla} \mathcal{L}^{m-1-j}\varepsilon_{m}(y,x),d\nu(y)\rangle$$
$$+\sum_{j=0}^{m-i-1}\int_{\partial\Omega}\mathcal{L}^{m-1-j}\varepsilon_{m}(y,x)\langle\widetilde{\nabla} \mathcal{L}^{j+i}v(y),d\nu(y)\rangle, \quad i=0,1,\ldots,m-1.$$

By passing to the limit as $x\rightarrow \partial\Omega$ from interior, we obtain the relations
\begin{equation}
\mathcal{L}^{i}v(x)\mid_{x\in\partial\Omega}=I_{i}(v)(x)\mid_{x\in\partial\Omega}=0,\quad
i=0,1,\ldots,m-1.
\end{equation}

Assuming for the moment the uniqueness of the solution of the boundary value problem
\begin{equation}\label{EQ:DPm}
\mathcal{L}^{m}v=0,
\end{equation}
$$
\mathcal{L}^{i}v\mid_{\partial\Omega}=0, \quad i=0,1,\ldots,m-1,
$$
we get that $v=u-w\equiv0$, for all $x\in\Omega$, i.e. $w$ coincides with $u$ in $\Omega$.
Thus \eqref{18} is the unique solution of the boundary value problem
\eqref{17}, \eqref{20} in $\Omega$.

It remains to show that the boundary value problem \eqref{EQ:DPm} has a unique solution
in $C^{2m}(\Omega)\cap C^{2m-1}(\overline{\Omega})$.
Denoting
$$\tilde v:= \mathcal{L}^{m-1}v,$$
this follows by induction
from the uniqueness in $C^{2}(\Omega)\cap C^{1}(\overline{\Omega})$ of the problem
$$
\mathcal{L} \tilde v=0,\quad
$$
$$
\tilde v\mid_{\partial\Omega}=0.
$$
The proof of Theorem \ref{THM:main2} is complete.
\end{proof}

\begin{rem}
It follows from Theorem \ref{THM:main2} that the kernel \eqref{19}, which is a
fundamental solution of the equation \eqref{17}, is the Green function of the
boundary value problem \eqref{17}, \eqref{20} in $\Omega$. Therefore, the
boundary value problem
\eqref{17}, \eqref{20} can serve as an example of an explicitly solvable
 boundary value problem for the iterated sub-Laplacian in any (admissible)
domain $\Omega$ on the homogeneous Carnot  group.
\end{rem}

\section{Refined Hardy inequality and uncertainty principles}
\label{Sec7}

Here we give another example of the use of the developed techniques to prove a refined version of the Hardy inequality for homogenous Carnot groups.
Let $u\in C^{\infty}_{0}(\mathbb{G}\backslash \{0\})$, $\alpha\in \mathbb{R}$, $Q\geq 3$, and $\alpha>2-Q.$ Then it was shown in \cite{GolKom} that we have the Hardy inequality
\begin{equation}\label{eqgHardy}
\int_{\mathbb{G}}d^{\alpha} |\nabla_{\mathbb G} u|^{2} \,d\nu\geq
\left(\frac{Q+\alpha-2}{2}\right)^{2}\int_{\mathbb{G}}
d^{\alpha}\frac{|\nabla_{\mathbb G} d|^{2}}{d^{2}}
|u|^{2}\,d\nu,
\end{equation}
and the constant $\left(\frac{Q+\alpha-2}{2}\right)^{2}$ is sharp.
In the case of ${\mathbb G}=(\mathbb R^{N},+)$, this recovers the Hardy inequality:
here $Q=N$, $d(x)=|x|$ is the Euclidean norm, and with $\alpha=0$, this gives
the classical Hardy inequality
\begin{equation}\label{HRn}
\int_{\mathbb{R}^{N}}|\nabla u(x)|^{2}dx\geq\left(\frac{N-2}{2}\right)^{2} \int_{\mathbb{R}^{N}}\frac{|u(x)|^{2}}{|x|^{2}}dx,\quad N\geq 3,
\end{equation}
where $\nabla$ is the standard gradient in $\mathbb{R}^{N}$, $u\in C_{0}^{\infty}(\mathbb{R}^{N}\backslash {0})$,
and the constant $\left(\frac{N-2}{2}\right)^{2}$ is sharp.
On the Heisenberg group versions of \eqref{eqgHardy} were obtained in
\cite{GL,GZ} using explicit formulae for the fundamental solution of $\mathcal L$ and of the
distance function $d$. We refer to \cite{DGP}, \cite{GolKom} and more recent work \cite{Lund09} to
for other references on this subject,
and to \cite{BCG} and \cite{BFG} for Besov space versions of Hardy inequalities on the
Heisenberg group and on graded groups, respectively.

We now present a refined local version of this inequality with an additional boundary term
on the right hand side of the inequality \eqref{eqgHardy}. The following Proposition \ref{LHardy} also implies inequality \eqref{eqgHardy} if we take the domain $\Omega$ containing the support of
$u\in C^{\infty}_{0}(\mathbb{G}\backslash \{0\})$, so that the boundary term in \eqref{LH2} is then equal to zero on $\partial\Omega$.
The (very short) proof of Proposition \ref{LHardy} relies on Green's first formula
that we obtained in Proposition \ref{green1}.

\begin{prop}\label{LHardy}
Let $\Omega\subset\mathbb{G}$ be an admissible domain with $0\not\in\partial\Omega$ and
let $\alpha\in \mathbb{R}$,\ $\alpha>2-Q$, and $Q\geq 3.$
Let  $u\in C^{1}(\Omega)\bigcap C(\overline{\Omega})$.
 Then the following generalised local Hardy
inequality is valid
\begin{multline}\label{LH2}
\int_{\Omega}d^{\alpha} |\nabla_{\mathbb G} u|^{2} \,d\nu\geq
\left(\frac{Q+\alpha-2}{2}\right)^{2}\int_{\Omega}
d^{\alpha}\frac{|\nabla_{\mathbb G} d|^{2}}{d^{2}}
|u|^{2}\,d\nu\\+
\frac{Q+\alpha-2}{2(Q-2)}\int_{\partial\Omega}d^{Q+\alpha-2}|u|^{2}
\langle\widetilde{\nabla}d^{2-Q},
d\nu\rangle,
\end{multline}
where $d$ is the $\mathcal{L}$-gauge and
$\nabla_{\mathbb G}=(X_{1},\ldots,X_{N_{1}})$.
\end{prop}

If $u=0$ on $\partial\Omega$, Proposition \ref{LHardy} reduces to \eqref{eqgHardy}.
Moreover, it follows from Remark \ref{REM:pos} and the range of $\alpha$ that the boundary term in \eqref{LH2} can be positive:
\begin{equation}\label{LH:pos}
\frac{Q+\alpha-2}{2(Q-2)}\int_{\partial\Omega}d^{Q+\alpha-2}|u|^{2}
\langle\widetilde{\nabla}d^{2-Q},
d\nu\rangle> 0.
\end{equation}
Since it is known that the constant $\left(\frac{Q+\alpha-2}{2}\right)^{2}$ in \eqref{LH2}
(or rather in \eqref{eqgHardy}) is sharp, the local inequality \eqref{LH2} gives a refinement
to \eqref{eqgHardy}.
We also note that in comparison to \eqref{eqgHardy}, we do not assume in Proposition  \ref{LHardy}
that $0$ is not in the support of the function $u$ since for $\alpha>2-Q$ all the integrals in
\eqref{LH2} are convergent.

Even if $0\in\partial\Omega$, the statement of Proposition \ref{LHardy} remains true if
$0\not\in\partial\Omega\cap{\rm supp }\, u$.

\begin{rem}\label{REM:pos}
Let us show that the last (boundary) term in \eqref{LH2} sometimes changes its
sign. For this purpose when $u=e^{-\frac{R}{2}d},\;R>0,$ applying Green's first
formula (see Proposition \ref{green1}) we calculate
$$\int_{\partial\Omega}d^{Q+\alpha-2}e^{-Rd}
\langle\widetilde{\nabla}d^{2-Q},
d\nu\rangle=\int_{\Omega}\widetilde{\nabla}(d^{Q+\alpha-2}e^{-Rd})d^{2-Q}d\nu$$
$$+\frac{1}{\beta_{d}}\int_{\Omega}d^{Q+\alpha-2}e^{-Rd}\mathcal{L}\beta_{d}d^{2-Q}d\nu=$$
$$=\int_{\Omega}\widetilde{\nabla}(d^{Q+\alpha-2}e^{-Rd})d^{2-Q}d\nu
=\sum_{k=1}^{N_{1}}\int_{\Omega}X_{k}(d^{Q+\alpha-2}e^{-Rd})X_{k}d^{2-Q}d\nu$$
$$=\sum_{k=1}^{N_{1}}\int_{\Omega}\left((Q+\alpha-2)d^{Q+\alpha-2-1}e^{-Rd}
X_{k}d-Rd^{Q+\alpha-2}e^{-Rd}X_{k}d\right)(2-Q)d^{2-Q-1}X_{k}d\,d\nu.$$
 Let
$\alpha=0$, $Q=3,$ and $\Omega\bigcap\Omega_{\frac{1}{R}}= \{\emptyset\},$
where $\Omega_{\frac{1}{R}}=\{x\in\mathbb{G}:\,d(x)<\frac{1}{R}\}$. Then we get
 $$\int_{\partial\Omega}de^{-Rd}
\langle\widetilde{\nabla}d^{-1},
d\nu\rangle=\sum_{k=1}^{N_{1}}\int_{\Omega}(Rd^{-1}-d^{-2})e^{-Rd}(X_{k}d)^{2}\,d\nu>0$$
can be positive, that is, we give an example which shows that the boundary term
can be positive. Similarly, if $u:=C=const$, then we get
$$\int_{\partial\Omega}d^{Q+\alpha-2}C^{2}
\langle\widetilde{\nabla}d^{2-Q},
d\nu\rangle$$
$$=C^{2}\sum_{k=1}^{N_{1}}\int_{\Omega}\left((Q+\alpha-2)d^{Q+\alpha-2-1}
X_{k}d\right)(2-Q)d^{2-Q-1}X_{k}d\,d\nu$$
$$=-C^{2}(Q+\alpha-2)(Q-2)\sum_{k=1}^{N_{1}}\int_{\Omega}d^{\alpha-2}(X_{k}d)^{2}d\nu<0,$$
which shows that the boundary term can also be negative.
\end{rem}

\begin{proof}[Proof of Proposition \ref{LHardy}]
Without loss of generality we can assume that $u$ is real-valued.
In this case, recalling that $$(\widetilde{\nabla}u)u=\sum_{k=1}^{N_{1}} (X_{k}u)X_{k}u=|\nabla_{\mathbb G}u|^{2},$$ \eqref{LH2}
reduces to
\begin{multline}\label{LH}
\int_{\Omega}d^{\alpha}(\widetilde{\nabla}u)u\,d\nu\geq
\left(\frac{Q+\alpha-2}{2}\right)^{2}\int_{\Omega}
d^{\alpha}\frac{(\widetilde{\nabla}d)d}{d^{2}}
u^{2}\,d\nu\\+
\frac{Q+\alpha-2}{2(Q-2)}\int_{\partial\Omega}d^{Q+\alpha-2}u^{2}
\langle\widetilde{\nabla}d^{2-Q},
d\nu\rangle,
\end{multline}
which we will now prove.
Setting $u=d^{\gamma}q$ for some $\gamma\not=0$ to be chosen later,
we have
$$(\widetilde{\nabla}u)u=
(\widetilde{\nabla}d^{\gamma}q)d^{\gamma}q=\sum_{k=1}^{N_{1}}X_{k}(d^{\gamma}q)X_{k}(d^{\gamma}q)$$
$$=\gamma^{2}d^{2\gamma-2}\sum_{k=1}^{N_{1}}(X_{k}d)^{2}q^{2}+2\gamma d^{2\gamma-1}q\sum_{k=1}^{N_{1}}X_{k}d\,X_{k}q+d^{2\gamma}
\sum_{k=1}^{N_{1}}(X_{k}q)^{2}$$
$$=\gamma^{2}d^{2\gamma-2}((\widetilde{\nabla}d)d)q^{2}+2\gamma d^{2\gamma-1}q(\widetilde{\nabla}d)q+d^{2\gamma}(\widetilde{\nabla}q)q.$$
Multiplying both sides of the above equality by $d^{\alpha}$ and applying Green's first formula (see Proposition \ref{green1}), we get
\begin{multline}\label{H1}
\int_{\Omega}d^{\alpha}(\widetilde{\nabla}u)ud\nu=\gamma^{2}
\int_{\Omega}d^{\alpha+2\gamma-2}((\widetilde{\nabla}d)d)\,q^{2}d\nu+
\frac{\gamma}{\alpha+2\gamma}\int_{\partial\Omega}q^{2}
\langle\widetilde{\nabla}d^{\alpha+2\gamma},
d\nu\rangle
\\-\frac{\gamma}{\alpha+2\gamma}\int_{\Omega}(\mathcal{L}d^{\alpha+2\gamma})q^{2}
d\nu+\int_{\Omega}d^{\alpha+2\gamma}(\widetilde{\nabla}q)qd\nu
\geq  \int_{\Omega}\gamma^{2}d^{\alpha+2\gamma-2}((\widetilde{\nabla}d)d)
\,q^{2}d\nu\\+
\frac{\gamma}{\alpha+2\gamma}\int_{\partial\Omega}q^{2}
\langle\widetilde{\nabla}d^{\alpha+2\gamma},
d\nu\rangle-\frac{\gamma}{\alpha+
2\gamma}\int_{\Omega}(\mathcal{L}d^{\alpha+2\gamma})q^{2}d\nu.
\end{multline}
On the other hand, it can be readily checked that we have
\begin{equation}\label{H2}
-\frac{\gamma}{\alpha+2\gamma}\mathcal{L}d^{\alpha+2\gamma}=-\gamma(\alpha+2\gamma+Q-2)d^{\alpha+2\gamma-2}(\widetilde{\nabla}d)d
-\frac{\gamma}{2-Q}d^{\alpha+2\gamma+Q-2}\mathcal{L}d^{2-Q}.
\end{equation}
Since $q^{2}=d^{-2\gamma}u^{2},$ substituting \eqref{H2} into \eqref{H1} we obtain

$$\int_{\Omega}d^{\alpha}(\widetilde{\nabla}u)ud\nu\geq (-\gamma^{2}-\gamma(\alpha+Q-2))\int_{\Omega}d^{\alpha}\frac{(\widetilde{\nabla}d)d}{d^{2}}u^{2}d\nu$$
$$-\frac{\gamma}{2-Q}\int_{\Omega}(\mathcal{L}d^{2-Q})d^{\alpha+Q-2}u^{2}dx
+
\frac{\gamma}{\alpha+2\gamma}\int_{\partial\Omega}d^{-2\gamma}u^{2}
\langle\widetilde{\nabla}d^{\alpha+2\gamma},
d\nu\rangle.$$
Recall that $\varepsilon=\beta_{d}d^{2-Q},\;Q\geq3,$ is the fundamental solution of the sub-Laplacian $\mathcal{L}$, therefore
$$\int_{\Omega}(\mathcal{L}d^{2-Q})d^{\alpha+Q-2}u^{2}dx=0,\;\alpha>2-Q,$$
independent of whether $0$ belongs to $\Omega$ or not, since $\mathcal L d^{2-Q}=\frac{1}{\beta_{d}}\delta$ in $\mathbb G$.
Thus
$$\int_{\Omega}d^{\alpha}(\widetilde{\nabla}u)ud\nu\geq (-\gamma^{2}-\gamma(\alpha+Q-2))\int_{\Omega}d^{\alpha}
\frac{(\widetilde{\nabla}d)d}{d^{2}}u^{2}d\nu+
\frac{\gamma}{\alpha+2\gamma}\int_{\partial\Omega}d^{-2\gamma}u^{2}
\langle\widetilde{\nabla}d^{\alpha+2\gamma},
d\nu\rangle.$$
Taking
$\gamma=\frac{2-Q-\alpha}{2},$
we obtain \eqref{LH}.
\end{proof}

Proposition \ref{LHardy} implies the following local
uncertainly principles.

\begin{cor}[Uncertainly principle on $\Omega$]\label{Luncertainty}
Let $\Omega\subset\mathbb{G}$ be an admissible domain with $0\not\in\partial\Omega$ and let $u\in C^{1}(\Omega)\bigcap C(\overline{\Omega})$.
Then
\begin{multline}\label{UP1}
\left(\int_{\Omega}d^{2}|\nabla_{\mathbb G}d|^{2} |u|^{2}d\nu\right)
\left(\int_{\Omega}|\nabla_{\mathbb G}u|^{2}d\nu\right)
\geq\left(\frac{Q-2}{2}\right)^{2}\left(\int_{\Omega}
|\nabla_{\mathbb G}d|^{2} |u|^{2} d\nu\right)^{2}\\+
\frac{1}{2}
\int_{\partial\Omega}d^{Q-2}|u|^{2}
\langle\widetilde{\nabla}d^{2-Q},
d\nu\rangle\left(\int_{\Omega}d^{2}|\nabla_{\mathbb G}d|^{2} |u|^{2}d\nu\right)
\end{multline}
and also
\begin{multline}\label{UP2}
\left(\int_{\Omega}\frac{d^{2}}{|\nabla_{\mathbb G}d|^{2}}|u|^{2}d\nu\right)
\left(\int_{\Omega}|\nabla_{\mathbb G}u|^{2}d\nu\right)
\geq\left(\frac{Q-2}{2}\right)^{2}\left(\int_{\Omega}|u|^{2}d\nu\right)^{2}
\\+
\frac{1}{2}
\int_{\partial\Omega}d^{Q-2}|u|^{2}
\langle\widetilde{\nabla}d^{2-Q},
d\nu\rangle\,\left(\int_{\Omega}\frac{d^{2}}{|\nabla_{\mathbb G}d|^{2}}|u|^{2}d\nu\right).
\end{multline}
\end{cor}
\begin{proof}
Again, assuming $u$ is real-valued, and taking $\alpha=0$ in the inequality \eqref{LH} we get
$$ \left(\int_{\Omega}d^{2}((\widetilde{\nabla}d)d)|u|^{2}d\nu\right)
\left(\int_{\Omega}(\widetilde{\nabla}u)ud\nu\right)\geq$$$$ \left(\frac{Q-2}{2}\right)^{2}\left(\int_{\Omega}d^{2}
((\widetilde{\nabla}d)d)|u|^{2}d\nu\right)\int_{\Omega}
\frac{(\widetilde{\nabla}d)d}{d^{2}}|u|^{2}\,d\nu$$$$+
\frac{1}{2}
\int_{\partial\Omega}d^{Q-2}|u|^{2}
\langle\widetilde{\nabla}d^{2-Q},
d\nu\rangle\left(\int_{\Omega}d^{2}|\nabla_{\mathbb G}d|^{2} |u|^{2}d\nu\right)$$
$$
\geq\left(\frac{Q-2}{2}\right)^{2}\left(\int_{\Omega}
((\widetilde{\nabla}d)d)|u|^{2}d\nu\right)^{2}$$$$+
\frac{1}{2}
\int_{\partial\Omega}d^{Q-2}|u|^{2}
\langle\widetilde{\nabla}d^{2-Q},
d\nu\rangle\left(\int_{\Omega}d^{2}|\nabla_{\mathbb G}d|^{2} |u|^{2}d\nu\right),
$$
where we have used the H\"older inequality in the last line.
This shows \eqref{UP1}. The proof of \eqref{UP2} is similar.
\end{proof}

By Remark \ref{REM:pos}, the last (boundary) terms in \eqref{UP1} and
\eqref{UP2} can be positive. Thus, they provide refinements to the uncertainty
principles on $\mathbb G$ when we take $u\in C_{0}^{\infty}(\mathbb
G\backslash\{0\})$ and then $\Omega$ large enough so that it contains the
support of $u$, so that these (boundary) terms on $\partial\Omega$ vanish. In
the Euclidean case ${\mathbb G}=(\mathbb R^{N},+)$ with $d(x)=|x|$, we have
$|\nabla d|=1$, so that both \eqref{UP1} and \eqref{UP2} imply the classical
uncertainty principle for $\Omega\subset\mathbb R^{N}$,
\begin{equation*}\label{UPRn}
\left(\int_{\Omega} |x|^{2} |u(x)|^{2}dx\right)
\left(\int_{\Omega}|\nabla u(x)|^{2}dx\right)
\geq\left(\frac{N-2}{2}\right)^{2}\left(\int_{\Omega}
 |u(x)|^{2} dx\right)^{2},\quad N\geq 3.
\end{equation*}

\section{Appendix. Boundary forms and measures}
\label{SEC:mes}

In this section we briefly describe the relation between the forms $\langle X_{j},d\nu\rangle$,
perimeter measure, and the surface measure on the boundary $\partial\Omega$.
We would like to thank Valentino Magnani for a discussion on this topic as well as Nicola Garofalo for useful comments.

Let $L_{X}$ denote the Lie derivative with respect to the vector field $X$.
The Cartan formula for $L_{X}$ gives
$$L_{X}=d\,\imath_{X}+\imath_{X}\,d,$$
where $$\imath_{X} d\nu =\langle X, d\nu \rangle$$ is the contraction of the volume form
$d\nu=dx_{1}\wedge\ldots\wedge dx_{n}$ by $X$.
For any left invariant vector field $X_{j}$ we have
\begin{equation}\label{div1}
\int_{\Omega}X_{j} \varphi \, d\nu=
\int_{\Omega}{\rm div}\,(\varphi X_{j})d\nu
=\int_{\Omega}L_{\varphi X_{j}}d\nu=\int_{\Omega}d(\imath_{\varphi X_{j}}d\nu)=
\int_{\partial\Omega} \varphi \langle X_{j}, d\nu\rangle,
\end{equation}
where we use the Cartan formula and
 the Stokes formula in the last two equalities, respectively.
 This explains Proposition \ref{stokes} from the general point of view of differential forms.
Now, recall that the perimeter measure on $\partial\Omega$, which we may assume piecewise smooth here for simplicity, is given by
$$
\sigma_{H}(\partial\Omega)=
\sup\left\{ \sum_{i=1}^{N_{1}}\int_{\partial\Omega}\psi_{i}
\langle X_{i}, d\nu\rangle:\; \psi=(\psi_{1},\ldots,\psi_{N_{1}}),\, |\psi|\leq 1,\,
\psi \in C^{1}\right\}.
$$
Denoting by
 $\langle\cdot, \cdot\rangle_{E}$ the Euclidean scalar product, we write
$$|v_{H}|:=\left(\sum_{j=1}^{N_{1}}\langle v, X_{j}\rangle_{E}^{2}\right)^{\frac{1}{2}}$$
and $$|v_{H}|_{j}:=
\frac{\langle v, X_{j}\rangle_{E}}{|v_{H}|},$$
where $v$ is a vector, which will be later the outer
unit (with respect to the horizontal stratum) vector on $\partial\Omega$.
If $dS$ is the surface measure on $\partial\Omega$, we have
$$
d\sigma_{H}=|v_{H}| dS,
$$
and all these relations are well-defined because the perimeter measure
of the set of characteristic points of a smooth domain $\Omega$ is zero.
We can now calculate
$$\int_{\Omega}X_{j} \varphi\, d\nu=\int_{\Omega} {\rm div} (\varphi X_{j})d\nu=\int_{\partial\Omega}
\varphi\,\imath_{X_{j}}(d\nu)=\int_{\partial\Omega}\varphi
\langle v, X_{j}\rangle_{E}dS$$$$=
\int_{\partial\Omega}\varphi
\frac{\langle v, X_{j}\rangle_{E}}{|v_{H}|}|v_{H}|dS
=\int_{\partial\Omega}\varphi |v_{H}|_{j} \,d\sigma_{H}.
$$
Combining this with \eqref{div1} we obtain
\begin{equation}\label{div2}
\int_{\partial\Omega}\varphi\langle X_{j}, d\nu \rangle=\int_{\partial\Omega}\varphi |v_{H}|_{j} \, d\sigma_{H}.
\end{equation}
Let us now assume that $X_{j}$ are orthonormal on $\mathfrak g$, and let
$$X=\sum_{j=1}^{N_{1}}f_{j}X_{j}.$$
We write $$v_{H}=\sum_{j=1}^{N_{1}}\langle v,X_{j}\rangle_{E}X_{j}$$ for a vector $v$ with
$|v_{H}|=1$. Then we have
$$\langle X,v_{H}\rangle_{\mathfrak g}=\sum_{j=1}^{N_{1}} f_{j} |v_{H}|_{j}.$$
Now, applying \eqref{div2} with $\varphi=f_{j}$ and summing over $j$, we get
\begin{equation}\label{div3}
\int_{\partial\Omega} \sum_{j=1}^{N_{1}} f_{j}\langle X_{j}, d\nu\rangle=
\int_{\partial\Omega} \langle X, v_{H}\rangle_{\mathfrak g} \,d\sigma_{H},
\quad X=\sum_{j=1}^{N_{1}}f_{j}X_{j},
\end{equation}
which gives a link between the form in \eqref{EQ:S2} and the perimeter measure
$d\sigma_{H}.$

\end{document}